\documentclass[11pt,a4paper]{article}
\usepackage[cp1251]{inputenc}
\usepackage[leqno]{amsmath}
\usepackage{amsthm}
\usepackage{amsfonts,amstext,amssymb,verbatim,epsfig}
\usepackage{dsfont}
\usepackage{enumerate}
\usepackage{psfrag}
\usepackage{graphicx}
\usepackage{color}
\usepackage{float}
\usepackage{subfig}

\usepackage[top=3cm, bottom=3.5cm, left=2.5cm, right=2.5cm]{geometry}

\def\R{{\mathbb{R}}}

\def\Z{{\mathbb{Z}}}
\def\P{{\mathbb{P}}}
\def\E{{\mathbb{E}}}

\def\blspace{{\mathcal{C}}} %space of continuous rooted loops
\def\bbm{{\mu^{\mathrm{bb}}}} %brownian bridge measure
\def\blm{{\mu^{\mathrm{bl}}}} %brownian loop measure
\def\bls{{\mathcal B}} %brownian loop soup
\def\rwlspace{{\mathcal L}} %space random walk loops
\def\rwlm{{\mu^{\mathrm{rwl}}}} %random walk loop measure
\def\rwls{{\mathcal R}} %random walk loop soup
\def\blrange{{\alpha}} %5/8 from LT
\def\loe{{\mathrm{LE}}} %loop erasure
\def\lew{{\mathrm{LEW}}} %rescaled lerw
\def\koz{{\mathrm{K}}} %scaling limit of lerw
\def\bm{{\mathrm{BM}}} %brownian motion
\def\ls{{\mathrm{LS}}} 
\def\bs{{\mathrm{BS}}}
\def\hd{{\mathrm{dim}_{\mathrm{H}}}} %hausdorff dimension
\def\ballR{{\mathbb D}}
\def\ql{{\mathrm{QL}}} %quasi loop
\def\es{{\mathrm{Es}}} %Es(m,n)

\newtheorem{theorem}{Theorem}[section]
\newtheorem{lemma}[theorem]{Lemma}%[section]
\newtheorem{proposition}[theorem]{Proposition}%[section]

\newtheorem*{proposition*}{Proposition}

\newtheoremstyle{likedef}
  {}%
  {}%
  {}%
  {}%{\parindent}%
  {\bfseries}%
  {.}%
  {.5em}%
  {}%

\theoremstyle{likedef}

\newtheorem{remark}[theorem]{Remark}%[section]
\newtheorem{claim}[theorem]{Claim}
\newtheorem{conjecture}[theorem]{Conjecture}

\numberwithin{equation}{section}

\begin{document}
\title{On Brownian motion, simple paths, and loops}

\author{
Artem Sapozhnikov\thanks{
University of Leipzig, Department of Mathematics, 
Augustusplatz 10, 04109 Leipzig, Germany.
email: artem.sapozhnikov@math.uni-leipzig.de}
\and Daisuke Shiraishi\thanks{Kyoto University, Department of Mathematics, Kyoto, Japan 
and Forschungsinstitut f\"ur Mathematik, ETH Z\"urich, R\"amistrasse 101, 8092 Z\"urich, Switzerland.
e-mail: daisuke@math.kyoto-u.ac.jp}
}

\maketitle

\begin{abstract}
We provide a decomposition of the trace of the Brownian motion  
into a simple path and an independent Brownian soup of loops that intersect the simple path. 
More precisely, we prove that any subsequential scaling limit of the loop erased random walk is a simple path (a new
result in three dimensions), which can be taken as the simple path of the decomposition.
In three dimensions, we also prove that the Hausdorff dimension of any such subsequential scaling limit lies in $(1,\frac53]$. 
We conjecture that our decomposition characterizes uniquely the law of the simple path. 
If so, our results would give a new strategy to the existence of the scaling limit of the loop erased random walk 
and its rotational invariance. 
\end{abstract}

\section{Introduction}

How does the Brownian motion in $\R^d$ look like? 
This question has fascinated probabilists and mathematical physicists for a long time, and it continues to be an unending source of challenging problems.
Not too long after the existence of the Brownian motion was rigorously shown by Wiener in 1923, 
L\'evy \cite{Lev} proved that a two dimensional Brownian motion intersects itself almost surely,  
Kakutani \cite{Kak} showed that a $d$-dimensional Brownian motion is almost surely a simple path when $d \ge 5$, 
Dvoretzky, Erd\H{o}s and Kakutani \cite{DEK} verified that a Brownian motion intersects itself in three but not in four dimensions almost surely.
Much later, Taylor and Fristedt \cite{Tay,Fri} found that the Hausdorff dimension of the set of double points of the Brownian motion is two in two dimensions and one in three dimensions. 

In this paper, we are interested in the nature of self-intersections, more specifically, how loops formed by the Brownian motion are distributed in space. 
Consequently, from our point of view, we may focus on the case of two and three dimensions.  
We give an explicit representation of such loops by establishing a decomposition of the Brownian path into independent simple path and a set of loops.
In order to explain it, let us begin with a similar problem for a simple random walk.

Consider a simple random walk (SRW) on the rescaled lattice $\frac1n\Z^d$ started at the origin and stopped upon exiting from the unit ball. 
Its loop erasure, the {\it loop erased random walk} (LERW), is a simple path connecting the origin with the complement of the ball obtained from the random walk path 
by chronologically erasing all its loops. 
Remarkably, the law of these loops is very explicit. 
They come from a Poisson point process of discrete loops (a {\it random walk loop soup}) on $\frac1n\Z^d$ independent from the loop erasure \cite[Section~9]{LL10}. 
More precisely, if we denote the loop erased random walk by $\lew_n$ and an independent random walk loop soup in the unit ball by $\ls_n$, 
the exact definitions will come later (see Sections~\ref{sec:lerw} and \ref{sec:rwls}), then 
\begin{equation}\label{eq:RWdecomposition:intro}
\begin{array}{c}
\text{the union of $\lew_n$ and the loops from $\ls_n$ intersecting $\lew_n$ has the same law as}\\
\text{the trace of a SRW on $\frac1n\Z^d$ started at $0$ and stopped upon exiting from the unit ball.}
\end{array}
\end{equation}

As the Brownian motion is the scaling limit of a simple random walk, 
it is natural to start by looking for an analogue of the random walk path decomposition in the continuous setting. 
However, unlike the random walk, the Brownian motion has a dense set of loops and it is not clear how to remove them in chronological order. 
Zhan proved in \cite{Zhan} the existence of a loop erasure of planar Brownian motion, but the uniqueness is missing, and 
three dimensions is for the time being out of reach. 
Nevertheless, we are able to get an analogue of \eqref{eq:RWdecomposition:intro} for the Brownian motion by passing suitably to the large-$n$ limit 
on the both sides of the decomposition \eqref{eq:RWdecomposition:intro}. 
First of all, after interpolating linearly, we may view all the lattice paths and loops as continuous curves and loops of $\R^d$, 
more usefully, as elements in the metric space of all compact subsets of the closed unit ball with the Hausdorff metric, denote it by $(\mathcal K_D, d_H)$. 
Let $\koz$ be any weak subsequential limit of $\lew_n$ in $(\mathcal K_D, d_H)$, and $\bs$ the limit of $\ls_n$, which turns out to be 
the Brownian loop soup of Lawler and Werner \cite{LW04}. 
\begin{theorem}\label{thm:K+BLS=BM}
In $2$ and $3$ dimensions, the union of $\koz$ and all the loops from an independent $\bs$ that intersect $\koz$ has the same law 
in $(\mathcal K_D, d_H)$ as the trace of the Brownian motion stopped on exiting from the unit ball. 
\end{theorem}
A related result has been proved in \cite{LSW2} for the ``filling'' of the planar Brownian path, 
where the filling of a closed set $A$ in $\R^2$ is the union of $A$ with all the bounded connected components of $\R^2 \setminus A$. 
It is shown there that the filling of the union of $\koz$ and the loops from $\bs$ intersecting $\koz$
has the same law as the filling of the Brownian path.
However, the filling of a random set does not characterize its law. 
For instance, the filling of the $\mathrm{SLE}_{6}$ started at $0$ up to the first exit time from the unit disc has the same law as the filling of the Brownian path 
\cite[Theorem 9.4]{LSW2}, while the law of $\mathrm{SLE}_{6}$ as a random compact subset of the disc is different from that of the Brownian trace. 

\medskip

In two dimensions, the sequence $\lew_n$ converges to the Schramm-Loewner evolution with parameter $2$ ($\mathrm{SLE}_2$) \cite{LSW}, 
a simple path \cite{Sch} with Hausdorff dimension $\frac54$ \cite{B}. 
In particular, Theorem~\ref{thm:K+BLS=BM} immediately gives a decomposition of the planar Brownian path into a simple path and loops. 
Unfortunately, no explicit expression for the scaling limit of the LERW is known or conjectured in three dimensions. 
Kozma \cite{K} proved that the sequence $\lew_{2^n}$ is Cauchy in $(\mathcal K_D, d_H)$, 
which gives the existence of the scaling limit as a random compact subset of the ball, 
and, topologically, this is all that has been known up to now. 
Our next main result shows that in three dimensions $\koz$ is a simple path. 
\begin{theorem}\label{thm:simplepath}
Let $\gamma^s$ and $\gamma^e$ be the end points of a simple path $\gamma$, and define 
\[
\Gamma = \left\{\gamma~:~\text{$\gamma$ is a simple path with $\gamma^s = 0$ and $\gamma\cap\partial D = \{\gamma^e\}$}\right\}.
\]
Then, almost surely, $\koz\in\Gamma$.
\end{theorem}
Theorems~\ref{thm:K+BLS=BM} and \ref{thm:simplepath} give a decomposition in $(\mathcal K_D, d_H)$ of the Brownian path 
into a simple path and loops also in three dimensions. 
For completeness, let us comment briefly on the higher dimensions. 
Our two main results also hold in higher dimensions, but the conclusions are rather trivial. --
In dimensions higher than $3$, the scaling limit of the LERW is a $d$-dimensional Brownian motion \cite[Theorem~7.7.6]{L},
it is itself a simple path, and the Brownian loop soup does not intersect it. 

We believe that the decomposition of the Brownian path into a simple path and loops as in Theorem~\ref{thm:K+BLS=BM} 
is important not only because it sheds light on the nature of self-intersections in the Brownian path, 
but more substatially, a uniqueness of the decomposition is expected, in which case, 
the law of $\koz$ would be uniquely characterized by the decomposition. 
\begin{conjecture}
Let $\koz_1$ and $\koz_2$ be random elements in $(\mathcal K_D, d_H)$ such that $\koz_1,\koz_2\in\Gamma$ almost surely,
and $\bs$ a Brownian loop soup in the unit ball independent from $\koz_1$ and $\koz_2$.  
If for each $i\in\{1,2\}$, the union of $\koz_i$ and all the loops from $\bs$ that intersect $\koz_i$   
has the same law as the trace of the Brownian motion stopped on exiting from the unit ball, 
then $\koz_1$ and $\koz_2$ have the same law in $(\mathcal K_D, d_H)$. 
\end{conjecture}
As immediate consequences of the uniqueness and Theorem~\ref{thm:K+BLS=BM}, one would 
get a new strategy to the existence of the scaling limit of the loop erased random walk and its rotational invariance. 
Needless to say, it would provide a description of the LERW scaling limit in three dimensions, which is still missing.
As far as we know, the conjecture has not been proved or disproved even in two dimensions.

\medskip

Any subsequential limit $\koz$ of $\lew_n$ is a simple path, and it is immediate that in three dimensions, $\koz$ has a different law than the Brownian path. 
In two dimensions, the law of $\koz$ is explicit, namely that of the trace of $\mathrm{SLE}_2$. 
Our final main result provides rigorous bounds on the Hausdorff dimension of $\koz$ in three dimensions. 
Let $\xi$ be the non-intersection exponent for $3$ dimensional Brownian motion \cite{L95}, 
and $\beta$ the growth exponent for the $3$ dimensional loop erased random walk \cite{Law99}. 
Both exponents exist by \cite{L95,Shi13} and satisfy the bounds $\xi\in(\frac12,1)$ and $\beta\in(1,\frac53]$, see \cite{L95,Law99}. 
\begin{theorem}\label{thm:hd}
In $3$ dimensions, $2-\xi\leq \hd(\koz)\leq \beta$ almost surely.
In particular, 
\[
1<\hd(\koz)\leq \frac53.
\]
\end{theorem}
The lower bound on $\hd(\koz)$ is an immediate application of Theorem~\ref{thm:K+BLS=BM} 
and a result on the Hausdorff dimension of the set of cut-points of the Brownian path. 
Here, a cut-point of a connected set $F$ that contains $0$ and intersects the boundary of the unit ball 
is any point $x\in F$ such that $0$ and the boundary of the ball are disconnected in $F \setminus \{x \}$. 
The Hausdorff dimension of the set of cut-points of the three-dimensional Brownian path from $0$ until exiting from the unit ball 
is precisely $2- \xi$, cf.~\cite{L95}. To see that it is a lower bound on $\hd(\koz)$, it remains to notice that in every decomposition of a Brownian path into a simple path and loops, all 
its cut-points are on the simple path. 

We expect that $\hd(\koz) = \beta$ almost surely. 
Some steps towards this equality will be made in \cite{Shi:wip}.
A same identity holds in two dimensions, where both the growth exponent and the Hausdorff dimension of $\mathrm{SLE}_2$ are 
known to be $\frac54$, see \cite{Ken,B} and also \cite{Law-2,M09}. 
In three dimensions, the value of $\beta$ is not known or conjectured. 
Numerical experiments suggest that $\beta = 1.62 \pm 0.01$ \cite{GB, Wil2}, 
but the best rigorous bounds are $1<\beta\leq \frac53$ \cite{Law99}.

\subsection{Some words about proofs}

Theorem~\ref{thm:K+BLS=BM} is an analogue of \eqref{eq:RWdecomposition:intro} in continuum. 
To prove it, we start with the decomposition of the random walk path \eqref{eq:RWdecomposition:intro} 
and suitably take scaling limits in both sides of the decomposition. 
By \eqref{eq:RWdecomposition:intro}, the union of the loop erased random walk $\lew_n$ and the loops from an independent random walk loop soup $\ls_n$ that intersect $\lew_n$ 
is the trace of simple random walk on the lattice $\frac1n\Z^d$ killed on exiting from the unit ball. 
In particular, it converges to the trace of the Brownian motion. 
On the other hand, $\lew_n$ converges weakly (along subsequences) to $\koz$, and, as was shown in two dimensions by Lawler and Trujillo \cite{LT04}, 
the loop soups $\ls_n$ and $\bs$ can be coupled so that with high probability there is a one-to-one correspondence between 
all large loops from $\ls_n$ and those from $\bs$, and each large loop from $\ls_n$ is very close, 
in the Hausdorff distance, to the corresponding loop in $\bs$. 
Such a strong coupling of loop soups can be extended to all dimensions with little effort, see Theorem~\ref{thm:coupling:soups}. 
So where is the challenge?

First, we may assume that $\lew_n$ and $\koz$ are defined on the same probability space and $d_H(\lew_n,\koz)\to 0$ almost surely. 
Let $\varepsilon<\delta$, and consider the event that $d_H(\lew_n,\koz)<\varepsilon$ and 
to each loop $\ell_n$ from $\ls_n$ of diameter at least $\delta$ corresponds a unique loop $\ell$ from the Brownian soup 
so that $d_H(\ell_n,\ell)<\varepsilon$. By the strong coupling of loop soups (see Theorem~\ref{thm:coupling:soups}), 
this event has high probability for all large $n$. 
The challenge is to show that the correspondence of loops in the strong coupling makes the right selection of Brownian loops. 
What may go wrong? If a loop $\ell_n\in\ls_n$ intersects $\lew_n$, then the corresponding Brownian loop $\ell$ 
does not have to intersect $\koz$, and vice versa. 
The meat of the proof is then to show that this does not happen. 

To demonstrate a difficulty, notice that very little is known about $\koz$ in three dimensions. 
In particular, it is not a priori clear if the Brownian soup really intersects $\koz$ almost surely (or even with positive probability). 
As we know, it is not the case in dimensions $4$ and higher, 
and the three dimensional Brownian soup does not intersect a line. 
As a result, not all paths in $\R^3$ are hittable by Brownian loops, so we have to show that $\koz$ is hittable. 
Moreover, we want that every Brownian loop of large diameter (bigger than $\delta$) that 
gets close enough (within $\varepsilon$ distance) to $\koz$ intersects it locally, and 
we want the same to be true for large random walk loops and $\lew_n$.

In two dimensions, analogous questions are classically resolved with a help of the Beurling projection principle, see \cite{K87}, 
which states that a random walk starting near any simple path will intersect it with high probability.
As we have just seen, such a principle cannot work in three dimensions for all paths. 
The main novelty of our proof is a Beurling-type estimate for the loop erased random walk 
stating that most of the samples of the LERW are hittable with probability close to one by an independent simple random walk 
started anywhere near the LERW, see Theorem~\ref{thm:beurling}. 
This result is then easily converted into an analogous statement for random walk loops, namely, with high probability, 
the only large loops that are close to $\lew_n$ are those that intersect it, see Proposition~\ref{prop:loopsatlew}.

\medskip

Similar complications arise when we try to show that $\koz$ is a simple path, although now without loop soups. 
We need to rule out a possibility that $\lew_n$ backtracks from far away. 
In his proof that $\koz$ is a simple path in two dimensions \cite{Sch}, Schramm introduced $(\varepsilon,\delta)$-quasi-loops 
as subpaths of $\lew_n$ ending within $\varepsilon$ distance from the start but stretching to distance $\delta$. 
Of course, if a quasi-loop exists for all large $n$, it collapses, in the large-$n$ limit, 
into a proper loop in $\koz$. 
Thus, to show that $\koz$ is a simple path, we need to rule out the existence of quasi-loops uniformly in $n$, namely, 
to show that for all $\delta>0$, 
\[
\lim_{\varepsilon\to0}\P\left[\text{$\lew_n$ does not have a $(\varepsilon,\delta)$-quasi-loop}\right] = 0,\quad\text{uniformly in $n$}.
\]
Schramm proved this in two dimensions using the Beurling projection principle \cite[Lemma~3.4]{Sch}. 
As remarked before, the principle does not longer work in three dimensions, but our Beurling-type esitmate 
is strong enough to get the desired conclusion, see Theorem~\ref{thm:ql}. 

We should mention that Kozma \cite{K} proved that with high probability (as $n\to\infty$), 
$\lew_n$ does not contain $(n^{-\gamma},\delta)$-quasi-loops, see \cite[Theorem~4]{K}. 
This was enough to establish the convergence of $\lew_n$'s, but 
more is needed to show that $\koz$ is a simple path. 
Unfortunately, Kozma's proof strongly relies on the fact that the choice of $\varepsilon$ is $n$-dependent, 
and we need to establish a new method to get the uniform estimate. 

\subsection{Structure of the paper}

The main definitions are given in Section~\ref{sec:defnothist}, the loop erased random walk and its scaling limit in Subsection~\ref{sec:lerw}, 
the random walk loop soup in Subsection~\ref{sec:rwls}, and the Brownian loop soup in Subsection~\ref{sec:bls}. 
In each subsection, we also discuss some properties and a few historical facts about the models. 
Subsection~\ref{sec:bls} also contains the statement about the coupling of the random walk and the Brownian loop soups that we use in the proof of Theorem~\ref{thm:K+BLS=BM}
(see Theorem~\ref{thm:coupling:soups}). 
Some notation that we only use in the proofs are summarized in Subsection~\ref{sec:furthernotation}.

The Beurling-type estimate for the loop erased random walk is given in Section~\ref{sec:beurlingtypeestimate} (see Theorem~\ref{thm:beurling}). 
Some related lemmas about hittability of the LERW are also stated there and may be of independent interest (see Lemmas~\ref{l:kozma:4.6} and \ref{l:4.6}). 
The proof of the Beurling-type estimate is given in Subsection~\ref{sec:proofofbeurling}. 
The rest of the section is devoted to the proof of an auxiliary lemma, and may be omitted in the first reading. 

In Section~\ref{sec:strongcoupling} we construct the coupling of the loop soups satisfying conditions of Theorem~\ref{thm:coupling:soups}. 
This section may be skipped in the first reading. 

The proof of our first main result, Theorem~\ref{thm:K+BLS=BM}, is contained in Section~\ref{sec:BMdecompositionproof}. 
It is based on Theorems~\ref{thm:coupling:soups} and \ref{thm:beurling}. 

In Section~\ref{sec:simplepath}, we prove that the scaling limit of the LERW is a simple path. 
In Subsection~\ref{sec:quasiloops}, we define quasi-loops and prove that the LERW is unlikely to contain them. 
The proof of our second main result, Theorem~\ref{thm:simplepath}, is given in Subsection~\ref{sec:simplepath:proof}. 
It is based on the quasi-loops-estimates from Subsection~\ref{sec:quasiloops}, namely on Propositions~\ref{prop:loopsatlew} and \ref{prop:loopsatboundary}.

Finally, in Section~\ref{sec:Hausdorffdimension} we prove bounds on the Hausdorff dimension of the scaling limit of the LERW stated in Theorem~\ref{thm:hd}. 
This proof is largely based on some earlier results on non-intersection probabilities for independent LERW and SRW obtained in \cite{Shi13}.
We recall these results in Subsection~\ref{sec:Hausdorffdimension:preliminaries} (see \eqref{eq:es:bounds}) and also prove there some of their consequences. 
The upper and lower bounds on the Hausdorff dimension are proved in the remaining subsections.

\section{Definitions, notation, and some history}\label{sec:defnothist}

\subsection{Loop erased random walk and its scaling limit}\label{sec:lerw}

We consider the graph $\Z^d$ with edges between nearest neighbors. 
If $x$ and $y$ are nearest neighbors in $\Z^3$, we write $x\sim y$. 
A path is a function $\gamma$ from $\{1,\ldots,n\}$ to $\Z^d$ for some $n\geq 1$ such that 
$\gamma(i)\sim\gamma(i+1)$ for all $1\leq i\leq n-1$. The integer $n$ is the length of $\gamma$, we denote it by $\mathrm{len}\, \gamma$. 
The {\it loop erasure} of a path $\gamma$, $\loe(\gamma)$, is the (simple) path obtained by removing loops from $\gamma$ in order of their appearance, namely,
\[
\begin{array}{rll}
\loe(\gamma)(1) &= &\gamma(1)\\[5pt]
\loe(\gamma)(i+1) &= &\gamma(j_i+1)\quad \text{if}\quad j_i = \max\{j:\gamma(j) = \loe(\gamma)(i)\}<\mathrm{len}\,\gamma.
\end{array}
\]
We are interested in the loop erasure of a simple random walk path started at $0$ and stopped at the first time when it exits from a large Euclidean ball, 
the {\it loop erased random walk} (LERW). 
The simple random walk started at $x\in\Z^d$ is a Markov chain $\{R(t)\}_{t\in\Z_+}$ with $R(0) = x$ and transition probabilities 
\[
\P[R(t+1) = z~|~R(t) = y] = \left\{\begin{array}{ll} \frac{1}{2d} & \text{if }z\sim y \\ 0 &\text{otherwise.}\end{array}\right.
\]
We denote its law and the expectation by $\P^x$ and $\E^x$, respectively. 

\medskip

LERW was originally introduced \cite{Law0} and studied extensively by Lawler (see \cite{Law99} and the references therein),
who considered LERW as a substitute for the self-avoiding walk (see \cite{Sla}), which is harder to analyze. 
Since its appearance, the LERW has played an important role both in statistical physics and mathematics through its relation to the uniform spanning tree (UST). 
Pemantle \cite{Pem} proved that paths in the UST are distributed as LERWs, 
furthermore, the UST can be generated using LERWs by Wilson's algorithm \cite{Wil}. 

\medskip

We are interested in the scaling limit of the LERW and its connections to the Brownian motion. 
Let $|\cdot|$ be the Euclidean norm in $\R^d$. 
The open ball of radius $r$ is defined as $D_r = \{x\in\R^d~:~|x| < r\}$, and we denote its closure by $\overline D_r$. 
When $r=1$, we just write $D$ and $\overline D$. 
We consider the loop erasure of the simple random walk path on $\Z^d$ from $0$ until the first exit time from $\overline D_n$, 
rescale it by $\frac1n$, and denote the corresponding simple path on the lattice $\frac1n\Z^d$ 
and its linear interpolation by $\lew_n$. 
Consider the metric space $(\mathcal K_D, d_H)$ of all compact subsets of $\overline D$ with the Hausdorff metric. 
We can think of $\lew_n$ as random elements of $\mathcal K_D$. 
Let $\mathrm{P}_n$ be the probability measure on $(\mathcal K_D, d_H)$ induced by $\lew_n$. 
Since $(\mathcal K_D, d_H)$ is compact and the space of Borel probability measures on a compact space is compact in the weak topology, 
for any subsequence $n_k$, we can find a further subsequence $n_{k_i}$ such that $\mathrm{P}_{n_{k_i}}$ converges weakly 
to a probability measure supported on compact subsets of $\overline D$. 
In fact, more is known. 
In two dimensions, $\lew_n$ converges weakly to $\mathrm{SLE}_2$ \cite{LSW} (actually, even in a stronger sense). 
In $3$ dimensions, $\lew_{2^n}$ converges weakly as $n\to\infty$ to a random compact subset of $\overline D$, 
invariant under rotations and dilations \cite{K}. 

The existence of the LERW scaling limit will not be used in this paper. 
In fact, as discussed in the introduction, we are hoping that our approach can give an alternative proof of the existence. 
All our results are valid for any subsequential limit of $\lew_n$, which we denote by $\koz$ throughout the paper, and 
we will write for simplicity of notation that $\lew_n$ converges to $\koz$ without specifying a subsequence.

\subsection{Random walk loop soup}\label{sec:rwls}

To have a useful description of the loops generated by the loop erasure of a random walk path, 
we define a Poisson point process of discrete loops. 

A rooted loop of length $2n$ in $\Z^d$ is a $(2n+1)$-tuple $\gamma = (\gamma_0,\ldots,\gamma_{2n})$ with $|\gamma_i - \gamma_{i-1}| = 1$ and $\gamma_0 = \gamma_{2n}$. 
Let $\rwlspace$ be the space of all rooted loops. 
We are interested in a Poisson point process of rooted loops in which each individual loop ``looks like'' a random walk bridge. 
We define the random walk {\it loop measure} $\rwlm$ as a sigma finite measure on $\rwlspace$ giving 
the value $\frac{1}{2n}\cdot\left(\frac{1}{2d}\right)^{2n}$ to each loop of length $2n$. 
The factor $\frac{1}{2n}$ should be understood as choosing the root of the loop of length $2n$ uniformly.
The {\it random walk loop soup} $\rwls$ is the Poisson point process on the space $\rwlspace\times(0,\infty)$ 
with the intensity measure $\rwlm\otimes\mathrm{Leb}_1$. 
For each $\lambda>0$, the random walk loop soup induces the Poisson point process on the space $\rwlspace$ 
with the intensity measure $\lambda\rwlm$, as a pushforward by the function 
$\sum_i \mathds{1}_{(\gamma_i,\lambda_i)} \mapsto \sum_{i:\lambda_i\leq \lambda}\mathds{1}_{\gamma_i}$.
We call the resulting process the {\it random walk loop soup of intensity} $\lambda$ and denote it by $\rwls^\lambda$. 

\medskip

Poisson ensembles of Markovian loops (loop soups) were introduced informally by Symanzik \cite{Sym} 
as a representation of the $\phi^4$ Euclidean field, and subsequently extensively researched 
in the physics community. 
The first rigorous definition of a loop soup was given by Lawler and Werner \cite{LW04} in the context of planar Brownian motion.
Our definition of the random walk loop soup is taken from \cite[Chapter 9]{LL10}.  
Random walk and Brownian loop soups have lately been an object of large attention from probabilists and mathematical physicists 
due to their intimate relations to the Gaussian free field, see, e.g., \cite{LeJ08,Szn12}. 
Of particular importance for us, is the following decomposition of a random walk path into its loop erasure and a 
collection of loops coming from an independent random walk loop soup of intensity $1$. 

\begin{proposition}\label{prop:RWdecomposition}\cite[Propositions~9.4.1 and 9.5.1]{LL10}
Let $L_n$ be the loop erasure of a simple random walk on $\Z^d$ started at $0$ and stopped upon exiting from $D_n$. 
Let $\rwls^1$ be an independent random walk loop soup, and denote by $R_n$ be the set of all loops (with multiplicities) from $\rwls^1$ 
that are contained in $D_n$ and intersect $L_n$. Then the union of $L_n$ and $R_n$ has the same law 
as the trace of a simple random walk on $\Z^d$ started at $0$ and stopped upon exiting from $D_n$. 
\end{proposition}
Our goal is to pass to the scaling limit in the above decomposition to get a similar representation for the Brownian path. 
The scaling limit of $L_n$ is a random compact subset of a unit ball, as discussed in the previous section. 
We will soon see that the scaling limit of a random walk loop soup is the Brownian loop soup of Lawler and Werner, which we introduce in the next section.

\medskip

We finish this section with a hands-on definition of the random walk loop soup. 
Let $\rwlm(z,n)$ be the restriction of $\rwlm$ to the loops of length $2n$ rooted at $z$. 
It is a finite measure with the total mass $\rwlm(z,n)[\rwlspace] = \frac{1}{2n}\, p_{2n}(z,z)$, where 
$p_{2n}(x,y)$ is the probability that the simple random walk started at $x$ will be at $y$ at step $2n$, 
and $\frac{\rwlm(z,n)}{\rwlm(z,n)[\rwlspace]}$ is the probability distribution of the random walk bridge of length $2n$ starting and ending at $z$. 
The measure $\rwlm$ can be expressed as a linear combination of probability measures on $\rwlspace$, 
\begin{equation}\label{def:rwlm}
\rwlm = \sum_{z\in\Z^d}\sum_{n\geq 1}\rwlm(z,n) = \sum_{z\in\Z^d}\sum_{n\geq 1}\frac{p_{2n}(0,0)}{2n}\cdot \frac{\rwlm(z,n)}{\rwlm(z,n)[\rwlspace]}\, ,
\end{equation}
which leads to the following simple recipe for sampling the random walk loop soups.  
Let 
\[
\widetilde N(z,n;\cdot),\quad n\in\{1,2,\ldots\}, z\in\Z^d,
\]
be independent Poisson point processes on $(0,\infty)$ with parameter $\frac{p_{2n}(0,0)}{2n}$. 
Let 
\[
\widetilde L(z,n;m),\quad n\in\{1,2,\ldots\}, z\in\Z^d, m\in\{1,2,\ldots\}, 
\]
be independent random walk bridges of length $2n$ starting and ending at $0$, independent of all the $\widetilde N(z,n;\cdot)$.
Then the multiset
\begin{equation}\label{def:rwls:2}
\left\{z+\widetilde L(z,n;m)~:~z\in\Z^d,~n\geq 1,~1\leq m\leq \widetilde N(z,n;\lambda)\right\} 
\end{equation}
is the random walk loop soup of intensity $\lambda$. 
In other words, we first generate the number of (labeled) random walk bridges of length $2n$, rooted at $z$, and 
with label at most $\lambda$, $\widetilde N(z,n;\lambda)$, and then sample their shapes according to the random walk bridge measure $\frac{\rwlm(z,n)}{\rwlm(z,n)[\rwlspace]}$.

\subsection{Brownian loop soup and a strong coupling of loop soups}\label{sec:bls}

Recall our strategy -- we want to get a decomposition of a Brownian path by taking a scaling limit of both sides in the 
corresponding random walk path decomposition. 
For this, we still need to discuss the existence of a scaling limit of the random walk loop soup. 
Actually, the scaling limit is explicit, it is the Brownian loop soup of Lawler and Werner \cite{LW04}, and we now give its description. 

\medskip

A rooted loop in $\R^d$ is a continuous function $\gamma:[0,t_\gamma]\to\R^d$ with $\gamma(0) = \gamma(t_\gamma)$, 
where $t_\gamma\in(0,\infty)$ is the time duration of $\gamma$. 
We denote by $\blspace$ the set of all rooted loops. 
For $z\in\R^d$ and $t>0$, let $\bbm(z,t)$ be the measure on $\blspace$ induced by 
the Brownian bridge from $z$ to $z$ of time duration $t$. 
The {\it Brownian loop measure} $\blm$ is the measure on $\blspace$ given by 
\[
\blm = \int_{\R^d}\int_0^\infty \frac{1}{t\cdot (2\pi t)^{\frac d2}}\, \bbm(z,t)\,dt\,dz.
\]
Notice the analogy with a similar representation \eqref{def:rwlm} of the random walk loop measure as a linear combination of random walk bridge measures. 
The measure $\blm$ of course inherits the invariance under the Brownian scaling, ($r\cdot \mathrm{space}, r^2\cdot \mathrm{time~duration}$), 
from the bridge measures. 

The {\it Brownian loop soup} $\bls$ in $\R^d$ is the Poisson point process on the space $\blspace\times(0,\infty)$ 
with the intensity measure $\blm\otimes \mathrm{Leb}_1$. 
For each $\lambda>0$, the Brownian loop soup induces the Poisson point process on the space $\blspace$ 
with the intensity measure $\lambda\blm$, as a pushforward by the function 
$\sum_i \mathds{1}_{(\gamma_i,\lambda_i)} \mapsto \sum_{i:\lambda_i\leq \lambda}\mathds{1}_{\gamma_i}$. 
We call the resulting process the {\it Brownian loop soup of intensity $\lambda$} and denote it by $\bls^\lambda$. 

\medskip

The Brownian loop soups exhibit strong connections with the Schramm-Loewner evolution and the Gaussian free field, see, e.g., \cite{Cam15} for an overview, 
and they have been quite extensively studied. The connection between the random walk loop soups and the Brownian ones has been shown by Lawler and Trujillo \cite{LT04} 
in two dimensions, who constructed a strong coupling between the two loop soups -- much more than needed to see that the scaling limit of a random walk loop soup is 
a Brownian soup. For our purposes, we need to extend the result of \cite{LT04} to higher dimensions.
Actually, only to dimension $3$, but we give an extension to arbitrary dimensions, 
as, on the one hand, the proof does not get more complicated, and, on the other, it may be instructive to see the dependence of various parameters on the dimension. 
Let
\begin{equation}\label{def:blrange}
\blrange = \frac{3d+4}{2d(d+2)}.
\end{equation}
\begin{theorem}\label{thm:coupling:soups}
There exist $C<\infty$ and a coupling of the Brownian loop soup $\bls = \{\bls^\lambda\}_{\lambda>0}$ and the random walk loop soup $\rwls = \{\rwls^\lambda\}_{\lambda>0}$
such that for any $\lambda>0$, $r\geq 1$, $N\geq 1$, and $\theta \in \left(\frac{2d}{d+4}, 2\right)$,
on the event of probability 
\[
\geq 1 - C\, (\lambda+1)\, r^d\, N^{-\min\left(\frac{d}{2},\,\theta(\frac d2+2)-d\right)}
\]
there is a one-to-one correspondence of random walk loops from $\rwls^\lambda$ of length $\geq N^\theta$ rooted in $[-rN,rN]^d$ 
and Brownian loops from $\bls^\lambda$ of length $\geq \frac{N^\theta-2}{d}+\blrange$ rooted in $[-rN-\frac12,rN+\frac12]^d$, 
such that the time durations of the corresponding loops differ by at most $\blrange$, 
and the supremum distance between the corresponding loops is $\leq C\, N^{\frac34}\log N$. 
Here, each discrete loop is viewed as a rooted loop in $\R^d$ after linear interpolation. 
\end{theorem}

\subsection{Further notation}\label{sec:furthernotation}

In this section, we summarize all the remaining notation that will be used at least in two different proofs. 
Those that are used only once are deferred until more appropriate spots. 

For $v\in\R^d$ and $r>0$, the (discrete) ball of radius $r$ centered at $v$ is 
the set 
\[
B(v,r) = \{x\in\Z^d:|x-v|\leq r\}.
\]
For $A\subset\Z^d$, we denote by $\partial A$ the exterior vertex boundary of $A$, namely, 
\[
\partial A = \{x\notin A:x\sim y\text{ for some $y\in A$}\}.
\]
We also define $\overline A = A\cup\partial A$. 
The boundary of a subset $V$ of $\R^d$ is denoted by $\partial_{\R^d}V$. 

\medskip

For a random walk $R$, we denote the hitting time of a set $A\subset\Z^d$ by $R$ by 
\[
T(A) = \inf\{t\geq 1:R(t)\in A\}. 
\]
For $v\in\R^d$ and $r>0$, we write 
\[
T_{v,r} = T(\partial B(v,r)).
\]
Quite often, we will consider two independent random walks on the same space. 
If so, we will denote these random walks by $R^1$ and $R^2$, 
their laws by $\P^{1,x_1}$ and $\P^{2,x_2}$ (where $x_i=R^i(0)$), and the corresponding hitting times by $T^i$ and $T^i_{v,r}$. 

\medskip

If $\gamma$ is a path, we denote by $\gamma[a,b]$ the path (or the set, depending on a situation) in $\Z^d$ consisting of the vertices $\gamma(a),\gamma(a+1),\ldots,\gamma(b)$.
If $\gamma_1$ and $\gamma_2$ are two paths in $\Z^d$ and $\gamma_1(\mathrm{len}\,\gamma_1)\sim\gamma_2(1)$, 
then we denote by $\gamma_1\cup\gamma_2$ the path of length $\mathrm{len}\,\gamma_1 + \mathrm{len}\,\gamma_2$ 
obtained by concatenating $\gamma_1$ and $\gamma_2$. 

\medskip

For a set $S\subset\R^d$ and $\epsilon>0$, we denote by $S^{+\epsilon}$ the $\epsilon$-neighborhood of $S$ and 
by $S^{-\epsilon}$ the subset of points of $S$ at distance $>\epsilon$ from the complement of $S$. 

\medskip

Finally, let us make a convention about constants. Large constants whose values are not important are denoted by $C$ and $C'$ and small ones by $c$ and $c'$. 
Their dependence on parameters varies from proof to proof. Constants marked with a subindex, e.g., $C_1$, $C_H$, $c_2$, $c_*$, 
keep their values within the proof where they appear, but will change from proof to proof.

\section{Beurling-type estimate}\label{sec:beurlingtypeestimate}

Throughout this section we assume that the dimension of the lattice is $3$. 
We prove that the loop erasure of a simple random walk is hittable with high probability by an independent random walk started 
anywhere near the loop erasure. 
\begin{theorem}\label{thm:beurling}
There exist $\eta>0$ and $C<\infty$ such that for any $\varepsilon>0$ and $n\geq 1$,  
\begin{equation}\label{eq:beurling}
\P^{1,0} \left[
\begin{array}{c}
\text{ For any  $x \in B(0,n)$ with $\mathrm{dist}\left( x, \loe \left( R^1[0, T^1_{0,n} ] \right) \right) \leq \varepsilon^2 n$,}\\[5pt]
\P^{2,x} \left[ R^2[0, T^2_{x,\sqrt{\varepsilon} n}] \cap   \loe \left( R^1[0, T^1_{0,n} ] \right)= \emptyset \right] \leq \varepsilon^{\eta}
\end{array}
\right]
\geq 1- C\varepsilon,
\end{equation}
see Section~\ref{sec:lerw} for the definition of $\loe$ and Section~\ref{sec:furthernotation} for the other notation. 
\end{theorem}
A result analogous to Theorem~\ref{thm:beurling} in $2$ dimensions is known as the Beurling projection principle, see \cite{K87}. 
It states that for any $\eta<\frac34$, the probability on the left hand side of \eqref{eq:beurling} equals to $1$.
In dimensions $d\geq 4$, the result of Theorem~\ref{thm:beurling} is not true. 

\medskip

Before moving on to the proof of Theorem~\ref{thm:beurling}, we discuss its main ingredients. They 
are of independent interest and also will be used in other proofs in this paper. 
First of all, from the point of view of this work, it would be enough to prove the estimate \eqref{eq:beurling} only for all those $x\in B(0,n)$ that are 
at least $\varepsilon n$ distance away from $0$ and the complement of $B(0,n)$. 
However, Theorem~\ref{thm:beurling} is a valuable tool in the study of loop erased random walks in three dimensions
and its applications will surely spread beyond the topics covered in this paper. 

The proof of Theorem~\ref{thm:beurling} is done by considering separately the cases when $x\in B(0,\varepsilon n)$ and $x\notin B(0,\varepsilon n)$. 
In the first case we use \cite[Lemma~4.6]{K},
which states that the LERW is hittable by an independent random walk in any wide enough annulus centered at the origin. 
\begin{lemma}\label{l:kozma:4.6}\cite[Lemma~4.6]{K}
For any $K\geq 1$, there exist $\eta>0$ and $C<\infty$ such that for all $r>s>1$, 
\[
\P^{1,0} \left[
\begin{array}{c}
\text{There exists $T\geq 0$ such that $\loe\left(R^1[0,T]\right)\nsubseteq B(0,r)$, and }\\[5pt]
\P^{2,0} \left[ R^2[0, T^2_{0,r}] \cap   \loe \left( R^1[0, T] \right)\cap \left(B(0,r)\setminus \overline{B(0,s)}\right)= \emptyset \right] > \left(\frac sr\right)^{\eta}
\end{array}
\right]
\leq C\, \left(\frac sr\right)^K.
\]
\end{lemma}
In the second case, we use an analogue of \cite[Lemma~4.6]{K} about hittability of the LERW in annuli that do not surround the origin. 
We give its proof in Section~\ref{sec:proof:l:4.6}. 
We will use this lemma also in Section~\ref{sec:simplepath} to show that the LERW scaling limit is a simple path. 
This is why we state here a slightly stronger result than we need for the proof of Theorem~\ref{thm:beurling}. 
We will comment more on this after stating the lemma. 
\begin{lemma}\label{l:4.6}
For any $K\geq 1$, there exist $\eta>0$ and $C<\infty$ such that for all $r>s>1$ and $v\notin B(0,r)$, 
\begin{equation}\label{eq:l:4.6}
\P^{1,0} \left[
\begin{array}{c}
\text{There exists $T\geq 0$ such that $\loe\left(R^1[0,T]\right)\nsubseteq B(v,s)^c$, and }\\[5pt]
\P^{2,v} \left[ R^2[0, T^2_{v,r}] \cap   \loe \left( R^1[0, T] \right)[0,\sigma]\cap \left(B(v,r)\setminus \overline{B(v,s)}\right)= \emptyset \right] > \left(\frac sr\right)^{\eta}
\end{array}
\right]
\leq C\, \left(\frac sr\right)^K,
\end{equation}
where $\sigma = \inf\{t\geq 0:\loe \left( R^1[0, T] \right)[0,t]\cap B(v,s)\neq\emptyset\}$.
\end{lemma}
As remarked above, the full strength of Lemma~\ref{l:4.6} will not be needed until Section~\ref{sec:simplepath}, 
where we reuse the lemma to prove that the LERW scaling limit is a simple path, see the proof of Claim~\ref{cl:ql:beurling}. 
In the proof of Theorem~\ref{thm:beurling} we will only apply a weaker version of \eqref{eq:l:4.6}, where
$\loe \left( R^1[0, T] \right)[0,\sigma]$ is replaced by $\loe \left( R^1[0, T] \right)$. 

\subsection{Proof of Theorem~\ref{thm:beurling}}\label{sec:proofofbeurling}

Without loss of generality, we may assume that $\varepsilon$ is small. 
The proof of Theorem~\ref{thm:beurling} is a simple consequence of Lemmas~\ref{l:kozma:4.6} and \ref{l:4.6}.
We estimate the probability in \eqref{eq:beurling} separately for $x$'s in and outside $B(0,\varepsilon n)$.
In the first case, we apply Lemma~\ref{l:kozma:4.6} to $T = T^1_{0,n}$, $s=2\varepsilon n$, $r = \frac12\sqrt{\varepsilon}n$, and $K = 2$, 
so that for some $\eta>0$ and $C<\infty$, 
\[
\P^{1,0} \left[
\P^{2,0} \left[ R^2[0, T^2_{0,\frac12\sqrt{\varepsilon} n}] \cap   \loe \left( R^1[0, T^1_{0,n} ] \right)
\cap\left(B(0,\frac12\sqrt{\varepsilon}n)\setminus\overline{B(0,2\varepsilon n})\right)= \emptyset \right] > \varepsilon^{\eta}
\right]
\leq C\varepsilon.
\]
By varying the starting point of $R^2$, we get the harmonic function in $B\left(0,2\varepsilon n\right)$,
\[
h(x)~:=~ \P^{2,x} \left[ R^2[0, T^2_{0,\frac12\sqrt{\varepsilon} n}] \cap   \loe \left( R^1[0, T^1_{0,n} ] \right)\cap \left(B(0,\frac12\sqrt{\varepsilon}n)\setminus\overline{B(0,2\varepsilon n})\right)= \emptyset \right].
\]
By the Harnack inequality \cite[Theorem~1.7.2]{L}, there exists a constant $C_H<\infty$ such that 
$h(x) \leq C_H\,h(0)$, for all $x\in B(0,\varepsilon n)$. In particular, 
\[
\P^{2,x} \left[ R^2[0, T^2_{0,\frac12\sqrt{\varepsilon} n}] \cap   \loe \left( R^1[0, T^1_{0,n} ] \right) = \emptyset \right]\leq C_H\, h(0).
\]
Since $B(x,\sqrt{\varepsilon}n) \supseteq B(0,\frac12\sqrt{\varepsilon}n)$ for all $x\in B(0,\varepsilon n)$, we also have 
\[
\P^{2,x} \left[ R^2[0, T^2_{x,\sqrt{\varepsilon} n}] \cap   \loe \left( R^1[0, T^1_{0,n} ] \right) = \emptyset \right]\leq C_H\, h(0).
\]
Plugging this into the very first inequality gives
\[
\P^{1,0} \left[
\begin{array}{c}
\text{For some $x\in B(0,\varepsilon n)$,}\\
\P^{2,x} \left[ 
R^2[0, T^2_{x,\sqrt{\varepsilon} n}] \cap   \loe \left( R^1[0, T^1_{0,n} ] \right) = \emptyset \right] > C_H\,\varepsilon^{\eta}
\end{array}
\right]
\leq C\varepsilon.
\]
This gives \eqref{eq:beurling} after slightly decreasing $\eta$. 

\medskip

It remains to consider the case $x\notin B(0,\varepsilon n)$. We prove that for some $\eta>0$ and $C<\infty$, 
\begin{equation}\label{eq:beurling:xnotin}
\P^{1,0} \left[
\begin{array}{c}
\text{ For some  $x \in B(0,n)\setminus \overline{B(0,\varepsilon n)}$ with $\mathrm{dist}\left( x, \loe \left( R^1[0, T^1_{0,n} ] \right) \right) \leq \varepsilon^2 n$,}\\[5pt]
\P^{2,x} \left[ R^2[0, T^2_{x,\varepsilon n}] \cap   \loe \left( R^1[0, T^1_{0,n} ] \right)= \emptyset \right] > \varepsilon^{\eta}
\end{array}
\right]
\leq C\varepsilon,
\end{equation}
which is slightly stronger than \eqref{eq:beurling}, since $T^2_{x,\sqrt{\varepsilon} n}$ of \eqref{eq:beurling} is replaced here by the smaller $T^2_{x,\varepsilon n}$. 
We start by covering $B(0,n)\setminus\overline{B(0,\varepsilon n)}$ by $s= 10\,\lfloor \varepsilon^{-6}\rfloor$
balls of radius $\varepsilon^2 n$ with centers at $v_1,\ldots, v_s\in B(0,n)\setminus\overline{B(0,\varepsilon n)}$. 
By the union bound, the right hand side of \eqref{eq:beurling:xnotin} is bounded from above by 
\[
\sum_{i=1}^s
\P^{1,0} \left[
\begin{array}{c}
\text{ There exists  $x \in B\left(v_i,\varepsilon^2 n\right)$ with $\mathrm{dist}\left( x, \loe \left( R^1[0, T^1_{0,n} ] \right) \right) \leq \varepsilon^2 n$,}\\[5pt]
\P^{2,x} \left[ R^2[0, T^2_{x,\varepsilon n}] \cap   \loe \left( R^1[0, T^1_{0,n} ] \right)= \emptyset \right] > \varepsilon^{\eta}
\end{array}
\right]
\]
For each $x\in  B\left(v_i,\varepsilon^2 n\right)$, 
$B(x,\varepsilon^2n)\subset B(v_i,2\varepsilon^2n)$ and 
$B(x,\varepsilon n)\supset B(v_i,\frac12\varepsilon n)$. 
Thus, the $i$th probability in the sum is at most 
\[
\P^{1,0} \left[
\begin{array}{c}
\text{$\loe \left( R^1[0, T^1_{0,n} ] \right)\nsubseteq B(v_i,2\varepsilon^2 n)^c$, and for some $x \in B\left(v_i,\varepsilon^2 n\right)$}\\[5pt]
\P^{2,x} \left[ R^2[0, T^2_{v_i,\frac12\varepsilon n}]\cap \loe \left( R^1[0, T^1_{0,n} ] \right)\cap
\left(B(v_i,\frac12\varepsilon n)\setminus \overline{B(v_i,2\varepsilon^2 n)}\right) = \emptyset \right] > \varepsilon^{\eta}
\end{array}
\right]
\]
By the Harnack inequality applied to the harmonic function
\[
\P^{2,x} \left[ R^2[0, T^2_{v_i,\frac12\varepsilon n}]\cap \loe \left( R^1[0, T^1_{0,n} ] \right)\cap
\left(B(v_i,\frac12\varepsilon n)\setminus \overline{B(v_i,2\varepsilon^2 n)}\right) = \emptyset \right], \quad
x \in B\left(v_i,2\varepsilon^2 n\right), 
\]
there exists a universal constant $c_H>0$ such that the $i$th probability is bounded from above by 
\[
\P^{1,0} \left[
\begin{array}{c}
\text{$\loe \left( R^1[0, T^1_{0,n} ] \right)\nsubseteq B(v_i,2\varepsilon^2 n)^c$, and}\\[5pt]
\P^{2,v_i} \left[ R^2[0, T^2_{v_i,\frac12\varepsilon n}]\cap \loe \left( R^1[0, T^1_{0,n} ] \right)\cap
\left(B(v_i,\frac12\varepsilon n)\setminus \overline{B(v_i,2\varepsilon^2 n)}\right) = \emptyset \right] > c_H\,\varepsilon^{\eta}
\end{array}
\right]
\]
Now, we apply Lemma~\ref{l:4.6} with $v=v_i$, $r = \frac12\varepsilon n$, $s=2\varepsilon^2n$, and $K=7$ 
to find $\eta>0$ and $C<\infty$ for which the above probability is $\leq C\varepsilon^7$. 
Thus, the probability from \eqref{eq:beurling:xnotin} is bounded from above by $(C\varepsilon^7)\cdot s \leq 10 C\,\varepsilon$. 
This proves \eqref{eq:beurling:xnotin} and completes the proof of Theorem~\ref{thm:beurling}.
\qed

\subsection{Proof of Lemma~\ref{l:4.6}}\label{sec:proof:l:4.6}

The scheme of the proof is conceptually the same as that of \cite[Lemma~4.6]{K}, 
except for the main improvement stated in Claim~\ref{c:4.3} below. 
For the reader's convenience and because of the importace of the result, we give a complete proof, 
which we organize in a sequence of claims. 
The first claim is a stronger version of \cite[Lemma~4.3]{K}, which is the first step in the proof of \cite[Lemma~4.6]{K}. 
This improvement is essentially the main reason why the remaining steps in the proof of \cite[Lemma~4.6]{K} can be adapted to our situation. 

\begin{claim}\label{c:4.3}
There exists $c_1>0$ such that for all $n>0$, $v\in\partial B(0,n)$, and $\Gamma\subset B(v,n)^c$, 
\[
\P^0\left[\mathrm{dist}\left(R(T_{0,n}), B(v,n)^c\right)\geq \frac n2 ~\Big|~ R[1,T_{0,n}]\cap \Gamma = \emptyset\right]\geq c_1.
\]
\end{claim}
\begin{proof}[Proof of Claim~\ref{c:4.3}]
An analogous claim for the random walk on $\Z^2$ is proved in \cite[Proposition~3.5]{M09}. 
The same scheme works for $\Z^3$ with a slightly more involved analysis of the corresponding harmonic function. 

\medskip

We begin with some auxiliary observations in $\R^3$. 
For $z\in \R^3$, let $\ballR(z)$ be the unit ball in $\R^3$ centered at $z$, and write $\ballR$ for $\ballR(0)$. 
Let $u\in\partial_{\R^3}\ballR$, $\delta>0$, and $M = \{z\in\partial_{\R^3}\ballR:|z-u|\leq \delta\}$. 
For $z\in\ballR$, let $h(z) = \mathbf P^z[W(\tau_\ballR)\in M]$, where $W$ is the standard Brownian motion in $\R^3$ 
and $\tau_\ballR$ is the first hitting time of $\partial_{\R^3}\ballR$ by $W$. Then $h$ is a harmonic function in $\ballR$ with 
the boundary condition $\mathds{1}_M$. In particular, it can be written as 
\[
h(z) = \frac{1}{4\pi}\int_M\frac{1 - |z|^2}{|z-\sigma|^3}d\sigma.
\]
We will need the following properties of $h$.
\begin{itemize}
\item
If $\delta$ is small enough, then for all $z\in \ballR\setminus\ballR(u)$, $h(z)\leq h(0)$. 
\begin{proof}
By the maximum principle, it suffices to consider $z\in\partial\ballR(u)\cap\ballR$. 
By the symmetry, it suffices to prove the claim for $u = (1,0,0)$ and $z = (z_1,z_2,0)$. 
Using geometric constraints, one can express $h(z)$ as a function of $z_2$ only,
\[
h(z) = f(z_2) = \frac{1}{4\pi}\int_M\frac{2\sqrt{1-z_2^2} - 1}{\left(3-2\sigma_1 + 2(\sigma_1 - 1)\sqrt{1-z_2^2} - 2\sigma_2z_2\right)^{\frac32}}d\sigma, \qquad z\in\partial\ballR(u)\cap\ballR.
\]
One can show by a direct computation that $z_2f'(z_2)\leq 0$ if $|\sigma_2|$ is sufficiently small, which proves the claim.  
\end{proof}
\item
Another direct computation gives $\frac{\partial h}{\partial u}(0) = \nu>0$ (derivative in the direction $u$) and 
$\frac{\partial h}{\partial u'} = 0$ for any $u'$ orthogonal to $u$. 
\item
There exists $r\in(0,1)$ such that for all $\delta\in(0,\frac14)$ and $r\leq |z|<1$ with $|z-u|\geq \frac12$, $h(z)\leq \frac14 h(0)$. 
This follows from the bound $\frac{1 - |z|^2}{|z-\sigma|^3}\leq 4^3(1-r^2)$.
\end{itemize}
Assume that $n$ is large enough so that $\overline{B(0,rn)}\subset n\ballR$. The function $h_n(z) = h(\frac zn)$ is harmonic in $n\ballR$. 
For $z\in\overline{B(0,rn)}$, let $\widetilde h_n(z) = \mathbb E^z[h_n(R(T_{0,rn}))]$ be the discrete harmonic function in $B(0,rn)$ which agrees with $h_n(z)$ on $\partial B(0,rn)$. 
By \cite[(1.23) and (1.34)]{L}, there exists $C<\infty$ such that for all $z\in B(0,rn)$, $|h_n(z) - \widetilde h_n(z)|\leq \frac{C}{n}$. 

\medskip

We proceed with the proof of the claim. 
Let $n\geq 1$ and $v\in \partial B(0,n)$. We choose $u = \frac{v}{|v|}\in\partial_{\R^3}\ballR$. 
Let $A=\frac{4C}{\nu}$ (with $\nu$ and $C$ as above) and $x\in B(0,rn)$ be such that $x_i = \lfloor Au_i\rfloor$. 
By Taylor's theorem, $h_n(x) - h_n(0)\geq \frac{A\nu}{2n}$ for large $n$. Thus, for any $z\in B(0,rn)\setminus B(v,n)$, 
\begin{eqnarray*}
\widetilde h_n(x) - \widetilde h_n(z)
&=&[\widetilde h_n(x) - h_n(x)] + [h_n(x) - h_n(0)] + [h_n(0) - h_n(z)] + [h_n(z) - \widetilde h_n(z)]\\
&\geq &-\frac{C}{n} + \frac{A\nu}{2n} + 0 - \frac{C}{n} \geq 0.
\end{eqnarray*}
Since $\Gamma\subset B(v,n)^c$, the same calculation as on \cite[page 1032]{M09} gives
\begin{eqnarray*}
\E^x[h_n(R(T_{0,rn}))~|~R[0,T_{0,rn}]\cap\Gamma=\emptyset]&\geq &\mathbb E^x[h_n(R(T_{0,rn}))]\\
&=&\widetilde h_n(x)\geq \widetilde h_n(0) \geq h(0) - \frac{C}{n} \geq \frac12 h(0).
\end{eqnarray*}
By splitting the above probability into the terms where $|R(T_{0,rn}) - v|$ is $\geq\frac{n}{2}$ and $<\frac{n}{2}$, and 
estimating $h_n(R(T_{0,rn}))$ from above by $\frac14 h(0)$ in the first case and by $1$ in the second,  
one gets exactly as on \cite[page 1033]{M09} that 
\[
\P^x\left[|R(T_{0,rn}) - v|\leq\frac{n}{2}~\big|~R[1,T_{0,rn}]\cap\Gamma=\emptyset\right]\geq c>0,
\]
which implies that $\P^0\left[|R(T_{0,n}) - v|\leq\frac{n}{2}~\big|~R[1,T_{0,n}]\cap\Gamma=\emptyset\right]\geq c'>0$. 
The proof of the claim is complete. 
\end{proof}

\medskip

Before we state the next claim, we introduce some notation. 
For a path $\gamma$ and $t\geq 1$, we define the set of cut points of $\gamma$ up to time $t$,
\[
\mathrm{cut}(\gamma;t) = \{\gamma(i)~:~i<t,~\gamma[1,i]\cap\gamma[i+1,\mathrm{len}\,\gamma] = \emptyset\}.
\]
Note that $\mathrm{cut}(\gamma;t)$ is non-decreasing in $t$, and non-increasing as $\gamma$ is extended. 
Also note that $\loe(\gamma)[1,\mathrm{len}\,\loe(\gamma)]\supset\mathrm{cut}(\gamma;\mathrm{len}\,\gamma)$. 

The following claim is an analogue of \cite[Lemma~4.4]{K}. 
\begin{claim}\label{c:4.4}
There exists $q>0$ such that the following holds. 
For any $\varepsilon>0$ there exist $\delta = \delta(\varepsilon)>0$ and $C = C(\varepsilon)<\infty$ such that 
for all $r>C$, $s\in[r(1+\varepsilon),2r]$, $\Gamma\subset B(0,s)^c$ with 
\[
\P^0\left[R[0,T_{0,4r}]\cap\Gamma\neq\emptyset\right]<\delta, 
\]
and $v\in\partial B(0,s)$, 
\[
\P^{1,v}\left[
\begin{array}{c}
\text{For all $y\in B(v,\varepsilon r)$,}\\[3pt]
\P^{2,y}\left[\mathrm{cut}\left(R^1[0,+\infty);T^1_{v,\varepsilon r}\right)\cap R^2[0,T^2_{0,4r}]\neq\emptyset\right]\geq q 
\end{array}
~\Big|~R^1[1,T_{0,4r}]\cap\Gamma=\emptyset
\right] \geq q.
\]
\end{claim}
\begin{proof}[Proof of Claim~\ref{c:4.4}]
Let $v\in\partial B(0,s)$, and define 
$\mathcal C = \mathrm{cut}\left(R^1[0,+\infty);T^1_{v,\varepsilon r}\right)$ and $A = B(v,\frac12\varepsilon r)\setminus\overline{B(v,\frac14\varepsilon r)}$. 
Let $\lambda>2$ to be fixed later, and take $\mu = \frac{\varepsilon}{4\lambda}$ and $\rho = \mu r$. 
By \cite[Lemma~4.2]{K}, there exists $c>0$ such that for all $x\in \partial B(v,\rho)$, 
\[
\P^{1,x}\left[\text{For all $y\in B(v,\frac18\varepsilon r)$,}\,\,
\P^{2,y}\left[\mathcal C\cap R^2[0,T^2_{0,4r}]\cap A\neq\emptyset\right]\geq c
\right]\geq c.
\]
Since the random walk started from any $y\in B(v,\varepsilon r)$ will hit $B(v,\frac18\varepsilon r)$ before exiting from $B(0,4r)$ with probability $>c'$, 
the previous inequality also holds for all $y\in B(v,\varepsilon r)$. Namely, 
there exists $c_2>0$ such that for all $x\in \partial B(v,\rho)$, 
\begin{equation}\label{eq:4.4:4.2}
\P^{1,x}\left[\text{For all $y\in B\left(v,\varepsilon r\right)$,}\,\,
\P^{2,y}\left[\mathcal C\cap R^2[0,T^2_{0,4r}]\cap A\neq\emptyset\right]\geq c_2
\right]\geq c_2.
\end{equation}
By \cite[Proposition~1.5.10]{L}, for any $z\in\partial B(v,\frac14\varepsilon r)$, 
\[
\P^{1,z}\left[T^1_{v,\rho}<\infty\right]\leq C'\,\frac{\rho}{\varepsilon r} = \frac{C_3}{\lambda}.
\]
Thus, 
\begin{equation}\label{eq:4.4:lambda}
\P^{1,v}\left[R^1[T^1_{v,\frac14\varepsilon r},+\infty)\cap B(v,\rho) \neq \emptyset\right]\leq \frac{C_3}{\lambda}. 
\end{equation}
Note that if the random walk $R^1$ started from $v$ does not return to $B(v,\rho)$ after $T^1_{v,\frac14\varepsilon r}$, 
then 
\begin{equation}\label{eq:4.4:cut}
\mathrm{cut}(R^1[0,+\infty);T^1_{v,\varepsilon r})\cap A = \mathrm{cut}(R^1[T^1_{v,\rho},+\infty);T^1_{v,\varepsilon r})\cap A.
\end{equation}
Denote by $M$ the set of points on $\partial B(v,\rho)$ which are at distance $\geq \frac{\rho}{2}$ from $B(0,s)^c$. 
By Claim~\ref{c:4.3}, 
\begin{equation}\label{eq:4.4:4.3}
\P^{1,v}\left[R^1(T^1_{v,\rho})\in M ~\Big|~ R^1[1,T^1_{v,\rho}]\cap \Gamma = \emptyset\right]\geq c_1.
\end{equation}
By the Harnack inequality applied to the harmonic function
\[
\P^{1,x}\left[R^1[0,T^1_{0,4r}]\cap\Gamma\neq\emptyset\right],\qquad x\in B(0,s-\frac{\rho}{2}),
\]
and the assumption on $\Gamma$, there exists $C_4=C_4(\varepsilon,\lambda)<\infty$ such that for any $x\in M$, 
\begin{equation}\label{eq:4.4:M:Gamma}
\P^{1,x}\left[R^1[0,T^1_{0,4r}]\cap\Gamma\neq\emptyset\right]\leq C_4\delta.
\end{equation}
All the ingredients are ready to conclude. We have
\begin{multline*}
\P^{1,v}\left[\text{For all $y\in B(v,\varepsilon r)$,}\,\,
\P^{2,y}\left[\mathcal C\cap R^2[0,T^2_{0,4r}]\neq\emptyset\right]\geq c_2~\Big|~
R^1[1,T^1_{0,4r}]\cap\Gamma=\emptyset
\right] \\
\geq
\frac{\P^{1,v}\left[
\begin{array}{c}
\text{For all $y\in B(v,\varepsilon r)$,}\,\,\P^{2,y}\left[\mathcal C\cap R^2[0,T^2_{0,4r}]\cap A\neq\emptyset\right]\geq c_2,\\[3pt]
R^1[1,T^1_{0,4r}]\cap\Gamma=\emptyset,~R^1(T^1_{v,\rho})\in M,~R^1[T^1_{v,\frac14\varepsilon r},+\infty)\cap B(v,\rho) = \emptyset
\end{array}\right]}
{\P^{1,v}\left[R^1[1,T^1_{v,\rho}]\cap \Gamma = \emptyset\right]}.
\end{multline*}
By the strong Markov property for $R^1$ at time $T^1_{v,\rho}$ and the identity \eqref{eq:4.4:cut}, 
the nominator of the above expression is bounded from below by 
\begin{multline*}
\P^{1,v}\left[R^1[1,T^1_{v,\rho}]\cap\Gamma=\emptyset,~R^1(T^1_{v,\rho})\in M\right]\\
\cdot \min_{x\in M}\,
\P^{1,x}\left[
\begin{array}{c}
\text{For all $y\in B(v,\varepsilon r)$,}\,\,\P^{2,y}\left[\mathcal C\cap R^2[0,T^2_{0,4r}]\neq\emptyset\right]\geq c_2,\\[3pt]
R^1[1,T^1_{0,4r}]\cap\Gamma=\emptyset,~R^1[T^1_{v,\frac14\varepsilon r},+\infty)\cap B(v,\rho) = \emptyset
\end{array}\right]
\end{multline*}
By \eqref{eq:4.4:4.2}, \eqref{eq:4.4:lambda}, and \eqref{eq:4.4:M:Gamma}, the above display is 
\[
\geq \left(c_2 - \frac{C_3}{\lambda} - C_4\delta\right)\cdot \P^{1,v}\left[R^1[1,T^1_{v,\rho}]\cap\Gamma=\emptyset,~R^1(T^1_{v,\rho})\in M\right],
\]
which implies that 
\begin{multline*}
\P^{1,v}\left[\text{For all $y\in B(v,\varepsilon r)$,}\,\,
\P^{2,y}\left[\mathcal C\cap R^2[0,T^2_{0,4r}]\neq\emptyset\right]\geq c_2~\Big|~
R^1[1,T^1_{0,4r}]\cap\Gamma=\emptyset
\right] \\
\geq
\left(c_2 - \frac{C_3}{\lambda} - C_4\delta\right)\cdot \P^{1,v}\left[R^1(T^1_{v,\rho})\in M~\Big|~R^1[1,T^1_{v,\rho}]\cap\Gamma=\emptyset\right]\\
\geq c_1\cdot \left(c_2 - \frac{C_3}{\lambda} - C_4\delta\right),
\end{multline*}
where the last inequality follows from \eqref{eq:4.4:4.3}. 
Finally we make a choice of parameters. We choose $\lambda$ so that $\frac{C_3}{\lambda}<\frac{c_2}{4}$. 
Then we choose $\delta$ so that $C_4\delta < \frac{c_2}{4}$ and $q = \frac{c_1c_2}{2}$. 
The proof of Claim~\ref{c:4.4} is complete. 
\end{proof}

\medskip

To state the next claim, we need more notation. 
A function $\gamma:\{1,\ldots,n\}\to \Z^3$ is called a discontinuous path of length $n$. 
All the definitions that we introduced for nearest neighbor paths extend without any changes to discontinuous paths. 
Given two discontinuous paths $\gamma_1$ and $\gamma_2$, we define the discontinuous paths $\loe_1(\gamma_1\cup\gamma_2)$ and 
$\loe_2(\gamma_1\cup\gamma_2)$ as follows. Let $t_* = \max\{t:\loe(\gamma_1)[1,t-1]\cap\gamma_2=\emptyset\}$. Then 
\[
\begin{array}{rll}
\loe_1(\gamma_1\cup\gamma_2) &= &\loe(\gamma_1\cup\gamma_2)[1,t_*],\\[3pt]
\loe_2(\gamma_1\cup\gamma_2) &= &\loe(\gamma_1\cup\gamma_2)\left[t_*+1,\mathrm{len}\,\loe(\gamma_1\cup\gamma_2)\right].
\end{array}
\]
The next claim is an analogue of \cite[Lemma~4.5]{K}.
\begin{claim}\label{c:4.5}
For any $\varepsilon>0$ and $\eta>0$, there exist $\delta>0$ and $C<\infty$ such that the following holds. 
For $r>0$, let $A_1 = B(0,2r)\setminus\overline{B(0,r)}$ and $A_2 = B(0,4r)\setminus\overline{B(0,\frac12r)}$. 
Then for any $r>C$, $s\in[r(1+\eta),2r]$, $v\in\partial A_1$, and 
a discontinuous path $\gamma\subset A_2$ with $\gamma(1)\in\partial B(0,4r)$ and $\gamma(\mathrm{len}\,\gamma)\sim v$, 
\begin{equation}\label{eq:4.5}
\P^{1,v}\left[
\begin{array}{c}
\loe_1(\gamma\cup R^1[0,T^1(\partial A_2)])\subset B(0,s)^c,\\[3pt]
\loe_2(\gamma\cup R^1[0,T^1(\partial A_2)])\cap B(0,s-\eta r)\neq \emptyset,\\[3pt]
\P^{2,0}\left[R^2[0,T^2_{0,4r}]\cap L'\neq\emptyset\right]<\delta
\end{array}
\right]<\varepsilon,
\end{equation}
where $L' = \loe(\gamma\cup R^1[0,T^1(\partial A_2)])[1,t]$ and $t = \min\{i:\loe(\gamma\cup R^1[0,T^1(\partial A_2)])(i)\in B(0,s-\eta r)\}$.
\end{claim}
\begin{proof}[Proof of Claim~\ref{c:4.5}]
Fix $\varepsilon>0$ and $\eta>0$. Take $q>0$ from Claim~\ref{c:4.4}. Let $K=K(\varepsilon)>2$ be an integer such that 
$(1-q)^K<\varepsilon$. Let $\varepsilon' = \frac{\eta}{2K}$ and $\delta' = \delta_{\mathrm{Claim}~\ref{c:4.4}}(\varepsilon')$ 
be the $\delta$ from Claim~\ref{c:4.4} corresponding to $\varepsilon_{\mathrm{Claim}~\ref{c:4.4}} = \varepsilon'$. 
Define $s_i = s - \frac{\eta i}{K+2}r$ for $i\in\{1,\ldots,K+1\}$. 
Let $j_k$ be as in the definition of the loop erasure, so that $R^1[j_k+1,T^1(\partial A_2)]$ is a random walk conditioned not to hit $\loe(\gamma\cup R^1[0,j_k])$. 
Let 
\[
\tau_i = \max\{j_k\leq T^1(\partial A_2)~:~\loe(\gamma\cup R^1[0,j_k])\subset B(0,s_i)^c\}.
\]
If $\tau_i<T^1(\partial A_2)$, then define $\Gamma_i = \loe(\gamma\cup R^1[0,\tau_i])$ and $v_i = R^1(\tau_i)$. 
By Claim~\ref{c:4.4}, 
\[
\P\left[\mathcal B_i~|~\tau_i,~R^1[0,\tau_i]\right]<1-q,
\]
where 
\[
\mathcal B_i = \left\{
\begin{array}{c}
\text{$\tau_i<T^1(\partial A_2)$, $\P^{2,0}\left[R^2[0,T^2_{0,4r}]\cap \Gamma_i \neq \emptyset\right]< \delta'$, and}\\[3pt]
\text{for some $y\in B(v_i,\varepsilon'r)$, $\P^{2,y}\left[\mathrm{cut}(R^1[\tau_i,+\infty);T^*_i)\cap R^2[0,T^2_{0,4r}]\neq\emptyset\right]\leq q$}
\end{array}
\right\},
\]
and $T^*_i = \min\{t>\tau_i:R^1(t)\in\partial B(v_i,\varepsilon'r)\}$. 
Note that the event $\mathcal B_i$ contains the following event
\[
\mathcal B_i' = \left\{
\begin{array}{c}
\text{$\tau_{i+1}<T^1(\partial A_2)$, $\P^{2,0}\left[R^2[0,T^2_{0,4r}]\cap \Gamma_i \neq \emptyset\right]< \delta'$, and}\\[3pt]
\text{for some $y\in B(v_i,\varepsilon'r)$, $\P^{2,y}\left[\mathrm{cut}(R^1[\tau_i,\tau_{i+1}];T^*_i)\cap R^2[0,T^2_{0,4r}]\neq\emptyset\right]\leq q$}
\end{array}
\right\},
\]
which depends only on $\tau_{i+1}$ and $R^1[0,\tau_{i+1}]$. Thus, 
\[
\P[\mathcal B_i'~|~\mathcal B_1',\ldots, \mathcal B_{i-1}'] 
=
\E\left[\P\left[\mathcal B_i'~|~\tau_i,~R^1[0,\tau_i]\right]~\big|~\mathcal B_1',\ldots, \mathcal B_{i-1}'\right] < 1 - q
\]
and $\P\left[\bigcap_{i=1}^{K}\mathcal B_i'\right]<(1-q)^K<\varepsilon$. 
It remains to show that the event in \eqref{eq:4.5} implies $\bigcap_{i=1}^{K}\mathcal B_i'$ for some choice of $\delta$. 
It is well known (see, e.g., \cite[Lemma~2.5]{K}) that there exists $c_* = c_*(\eta)>0$ such that 
\[
\P^0\left[T_{w_i,\varepsilon'r}<T_{0,4r}\right]\geq c_* \qquad\text{for all $i$ and $w_i\in\partial B(0,s_i)$}.
\]
Take $\delta<\min(\delta',c_*q)$. Then the event in \eqref{eq:4.5} implies $\bigcap_{i=1}^{K}\mathcal B_i'$. 
Indeed, if the event in \eqref{eq:4.5} occurs, 
\begin{itemize}\itemsep1pt
\item
since $\loe_1(\gamma\cup R^1[0,T^1(\partial A_2)])\subset B(0,s)^c$ and $\loe_2(\gamma\cup R^1[0,T^1(\partial A_2)])$ contains a path from $\partial B(0,s)$ to $\partial B(0,s-\eta r)$, 
all $\tau_i$'s are strictly smaller than $T^1(\partial A_2)$,
\item
since $\Gamma_i\subset L'$, $\P^{2,0}\left[R^2[0,T^2_{0,4r}]\cap \Gamma_i \neq \emptyset\right]< \delta < \delta'$, and
\item
since $\mathrm{cut}(R^1[\tau_i,\tau_{i+1}];T^*_i)\subset L'$, 
\begin{multline*}
\min_{y\in B(v_i,\varepsilon' r)}\,\P^{2,y}\left[\mathrm{cut}(R^1[\tau_i,\tau_{i+1}];T^*_i)\cap R^2[0,T^2_{0,4r}]\neq\emptyset\right]\\
\leq \frac{\P^{2,0}\left[\mathrm{cut}(R^1[\tau_i,\tau_{i+1}];T^*_i)\cap R^2[0,T^2_{0,4r}]\neq\emptyset\right]}
{\P^{2,0}\left[T^2_{v_i,\varepsilon'r}<T^2_{0,4r}\right]}<\frac{\delta}{c_*} < q.
\end{multline*}
\end{itemize}
The proof of Claim~\ref{c:4.5} is complete. 
\end{proof}

\bigskip

\begin{proof}[Proof of Lemma~\ref{l:4.6}]
Without loss of generality we may assume that $s$ and $r/s$ are big (possibly depending on $K$). 
Let $\rho_i = \frac{r}{2^i}$, for $i\in\{0,\ldots, I = \lfloor \log_2\frac rs\rfloor\}$, and 
consider the stopping times when the random walk $R^1$ jumps between different $\partial B(v,\rho_i)$'s, 
\[
\begin{array}{rll}
\tau_0' &= &T^1_{v,r}, ~~i'(0) = 0,\\[3pt]
\tau_j' &= &\min\left\{t>\tau_{j-1}'~:~ R^1(t)\in\partial B(v,\rho_{i'(j)})~\text{for some $i'(j)\neq i'(j-1)$}\right\}.
\end{array}
\]
The process $i'(j)$ dominates a random walk on $\{0,\ldots,I\}$ with a drift towards $0$ and an absorption at $0$. 
In particular, if $n_i' = \sharp\{j:i'(j)=i\}$, then there exists $\lambda = \lambda(K)$ such that 
$\P^{1,0}[\sum_{i=1}^{I} n_i' > \lambda I] \leq C(\frac sr)^K$. 
Consider the annuli 
\[
A_{1,i} = B(v,2\rho_{2i})\setminus\overline{B(v,\rho_{2i})}\quad\text{and}\quad 
A_{2,i} = B(v,4\rho_{2i})\setminus\overline{B(v,\frac12\rho_{2i})}, \qquad\text{for }1\leq i\leq \frac{I-1}{2}
\]
and the stopping times 
\[
\begin{array}{rll}
\tau_1 &= &T^1(\partial A_{1,1}), ~~i(0) = 1,\\[3pt]
\tau_j &= &\min\left\{t>\tau_{j-1}~:~ R^1(t)\in\partial A_{1,i(j)}~\text{for some $i(j)\neq i(j-1)$}\right\}.
\end{array}
\]
Note that the $\tau_j$ is a subsequence of the $\tau_j'$. Therefore, if $n_i = \sharp\{j:i(j)=i\}$, then 
$\P^{1,0}[\sum_{i=1}^{\lfloor \frac{I-1}{2}\rfloor} n_i > \lambda I] \leq C(\frac sr)^K$. 
Let $M = \lceil 6\lambda\rceil$, then 
\begin{equation}\label{eq:lemma4.2:1}
\P^{1,0}\left[\sharp\left\{1\leq i\leq \frac12 I~:~ n_i>M\right\}>\frac16I\right]\leq C\left(\frac sr\right)^K.
\end{equation}
For each $j$ and $m\in\{0,\ldots,M-1\}$, define $r_j = \rho_{2i(j)}$ and $s_{j,m} = 2r_j - \frac{m}{M}r_j$, 
discontinuous paths $\gamma_{i,j} = \loe(R^1[0,\tau_j-1]\cap A_{2,i})$ and $\gamma_j = \gamma_{i(j),j}$, 
the loop erasures $L_{k,j} = \loe_k(\gamma_j\cup R^1[\tau_j,\tau_{j+1}])$ for $k=1,2$, 
and the event $\mathcal X_{j,m}$ that 
\begin{itemize}
\item[(a)]
$L_{1,j}\subset B(v, s_{j,m})^c$,
\item[(b)]
$L_{2,j}\cap B(v, s_{j,m+1})\neq \emptyset$,
\item[(c)]
$\P^{2,v}\left[R^2[0,T^2_{v,4r}]\cap L_{j,m}'\neq\emptyset\right]<\delta_{\mathrm{Claim}~\ref{c:4.5}}$, 
where $L_{j,m}' = L_{1,j}\cup L_{2,j}[1,t]$ and $t = \min\{i:L_{2,j}(i)\in \partial B(v,s_{j,m+1})\}$. 
\end{itemize}
By Claim~\ref{c:4.5}, $\P^{1,0}[\mathcal X_{j,m}| \tau_j<\infty, R^1[0,\tau_j]]<\varepsilon_{\mathrm{Claim}~\ref{c:4.5}}$ and 
\[
\P^{1,0}\left[\cup_{m=0}^{M-1}\mathcal X_{j,m}~\big|~ \tau_j<\infty, R^1[0,\tau_j]\right]<M\,\varepsilon_{\mathrm{Claim}~\ref{c:4.5}}.
\]
Let $\mathcal X_j = \cup_{m=0}^{M-1}\mathcal X_{j,m}$. Then the sequence of their indicators is dominated by a sequence of independent Bernoulli 
random variables with parameter $M\,\varepsilon_{\mathrm{Claim}~\ref{c:4.5}}$. In particular, by choosing $\varepsilon_{\mathrm{Claim}~\ref{c:4.5}}$ 
small enough, 
\begin{equation}\label{eq:lemma4.2:2}
\P^{1,0}\left[\sharp\left\{j\leq \lambda I~:~\mathcal X_j\right\}>\frac{\lambda I}{M}\right]\leq C\left(\frac sr\right)^K.
\end{equation}

\medskip

To finish the proof of Lemma~\ref{l:4.6}, it suffices to show that 
the event in \eqref{eq:l:4.6} implies one of the events in \eqref{eq:lemma4.2:1} or \eqref{eq:lemma4.2:2}. 
We call an index $i$ good, if $n_i\leq M$ and none of the $\mathcal X_j$'s occur for $j$'s with $i(j) = i$. 
The following claim is essentially \cite[Sublemma~4.6.1]{K}. 
\begin{claim}\label{c:4.6.1}
Let $i$ be a good index. For $j>0$, let $n_{i,j} = \sharp\{k<j:R^1(\tau_k)\in\partial A_{1,i}\}$, 
$u_{i,j} = 2\rho_{2i} - \frac{n_{i,j}}{M}\rho_{2i}$, and $U_{i,j} = B(v,2\rho_{2i})\setminus\overline{B(v,u_{i,j})}$. 
Let $t_{i,j} = \inf\{t:\gamma_{i,j}(t)\cap \partial B(v,u_{i,j})\neq\emptyset\}$, and define $\gamma_{i,j}^* = \gamma_{i,j}[1,t_{i,j}]$ 
if $t_{i,j}<\infty$ and $\gamma_{i,j}^*=\emptyset$ otherwise. 
Then for all $j$, if $\gamma_{i,j}^*\neq\emptyset$, then 
\begin{equation}\label{eq:4.6.1}
\P^{2,v}\left[\gamma_{i,j}^*\cap U_{i,j}\cap R^2[0,T^2_{v,4\rho_{2i}}]\neq\emptyset\right]\geq \delta_{\mathrm{Claim}~\ref{c:4.5}}.
\end{equation}
\end{claim}
\begin{proof}[Proof of Claim~\ref{c:4.6.1}]
The same as the proof of \cite[Lemma~4.6.1]{K}. 
We use induction on $j$. If $R^1(\tau_j)\notin \partial A_{1,i}$, then $R^1[\tau_j,\tau_{j+1}]$ does not enter in $A_{1,i}$ 
and thus can only change $\gamma_{i,j}^*$ by erasing a piece from its end leading to $\gamma_{i,j+1}\cap \partial B(v,u_{i,j})=\emptyset$. 
(Note that in this case $u_{i,j} =u_{i,j+1}$.)
On the other hand, if $R^1(\tau_j)\in \partial A_{1,i}$, then $\mathcal X_j$ does not occur, and, in particular, $\mathcal X_{j,n_{i,j}}$ 
does not occur, meaning that one of (a), (b), or (c) fails. 
If (a) fails, then $\gamma_{i,j}^*\neq\emptyset$ and hence satisfies \eqref{eq:4.6.1}. 
Also, if $\gamma_{i,j+1}^*\neq\emptyset$, then it contains $\gamma_{i,j}^*$ and hence also satisfies \eqref{eq:4.6.1}.
If (a) holds but (b) fails, then $\gamma_{i,j+1}^*=\emptyset$. 
Finally, if (a) and (b) hold but not (c), then $\gamma_{i,j+1}^* = L'_{j,n_{i,j}}$, and \eqref{eq:4.6.1} holds. 
\end{proof}

\medskip

Assume that the event in \eqref{eq:l:4.6} occurs for some $T$. Let $J$ be such that $\tau_J\leq T < \tau_{J+1}$. 
Since $R^1[\tau_J,T]\subset A_{2,i(J)}$, Claim~\ref{c:4.6.1} shows that for each good index $i\neq i(J)$, either 
$\loe(R^1[0,T])$ does not cross $A_{1,i}$ from outside to inside or 
the first such crossing is hittable by $R^2$. 
Since we assume that the event in \eqref{eq:l:4.6} occurs, $\loe(R^1[0,T])$ crosses $B(v,r)\setminus\overline{B(v,s)}$ and, in particular, 
all $A_{1,i}$'s. Thus, for each good $i\neq i(J)$, 
\[
\P^{2,v}\left[\loe(R^1[0,T])[0,\sigma]\cap A_{1,i}\cap R^2[0,T^2_{v,4\rho_{2i}}]\neq\emptyset\right]\geq \delta_{\mathrm{Claim}~\ref{c:4.5}}.
\]
By the Harnack inequality applied to the harmonic function
\[
\P^{2,x}\left[\loe(R^1[0,T])[0,\sigma]\cap A_{1,i}\cap R^2[0,T^2_{v,4\rho_{2i}}]\neq\emptyset\right],\qquad x\in B(v,\frac12\rho_{2i}),
\]
there exists a constant $c_*>0$ (independent of $i$) such that 
\[
\P^{2,x}\left[\loe(R^1[0,T])[0,\sigma]\cap A_{1,i}\cap R^2[0,T^2_{v,4\rho_{2i}}]\neq\emptyset\right]\geq c_*\,\delta_{\mathrm{Claim}~\ref{c:4.5}},\qquad \text{for all }x\in B(v,\frac12\rho_{2i}).
\]

Recall that we aim to show that the event in \eqref{eq:l:4.6} implies one of the events in \eqref{eq:lemma4.2:1} or \eqref{eq:lemma4.2:2}.
Thus, assume that the both these events do not occur. 
This implies that at least $\frac16 I$ indices between $1$ and $\frac12 I$ are good. 
Let $n=\lfloor\frac{I}{12}\rfloor - 2$ and $i_1>i_2>\ldots >i_n>1$ be good indices with $i_k - i_{k+1}\geq 2$ and $i_k\neq i(J)$, then 
\begin{multline*}
\P^{2,v}\left[R^2[0,T^2_{v,r}]\cap\loe(R^1[0,T])[0,\sigma]\cap \left(B(v,r)\setminus \overline{B(v,s)}\right)=\emptyset\right]\\
\leq
\prod_{k=1}^{n-1}\P^{2,v}\left[R^2[T^2_{v,\frac12\rho_{2i_k}},T^2_{v,4\rho_{2i_k}}]\cap\loe(R^1[0,T])[0,\sigma]=\emptyset~\Big|~ R^2[0,T^2_{v,4\rho_{2i_{k+1}}}]\cap\loe(R^1[0,T])[0,\sigma]=\emptyset\right]\\
\leq 
\prod_{k=1}^{n-1}\max_{w\in B(v,\frac12\rho_{2i_k})}\P^{2,w}\left[R^2[0,T^2_{v,4\rho_{2i_k}}]\cap\loe(R^1[0,T])[0,\sigma]=\emptyset\right]\\
\leq 
(1 - c_*\, \delta_{\mathrm{Claim}~\ref{c:4.5}})^{n-1}\leq 
(1 - c_*\, \delta_{\mathrm{Claim}~\ref{c:4.5}})^{\frac{I}{12} - 4} 
< \left(\frac sr\right)^\eta,
\end{multline*}
for $\eta$ small enough. This contradicts the assumption that the event \eqref{eq:l:4.6} occurs. 
Thus, we have shown that if both the events \eqref{eq:lemma4.2:1} and \eqref{eq:lemma4.2:2} do not occur, then 
the event \eqref{eq:l:4.6} also does not occur. The proof of Lemma~\ref{l:4.6} is complete. 
\end{proof}

\section{Strong coupling of loop soups}\label{sec:strongcoupling}

In this section we prove Theorem~\ref{thm:coupling:soups} 
by giving an explicit coupling of the random walk and the Brownian loop soups satisfying the conditions of the theorem. 
In two dimensions, such coupling is obtained in \cite{LT04}. 
Our construction is an adaptation of the one from \cite{LT04} to higher dimensions. 

\subsection{Preliminaries}

In this section we prove two lemmas needed for the proof of Theorem~\ref{thm:coupling:soups}. 
\begin{lemma}\label{l:p2n:expansion}
For any $d\geq 2$, 
\[
p_{2n}(0,0) = 2\, \left(\frac{d}{4\pi n}\right)^{\frac d2}\cdot \left(1 - \frac{d}{8n} + O\left(\frac{1}{n^2}\right)\right),
\]
as $n\to\infty$. 
\end{lemma}
\begin{proof}
Using the inverse Fourier transform, one can write 
\[
p_{2n}(0,0) = \frac{1}{(2\pi)^d}\,\int_{[-\pi,\pi]^d}\left(\frac1d\,\sum_{k=1}^d\cos(\theta_k)\right)^{2n}\,d\theta.
\]
The result follows by analysing this integral, see, for instance, \cite{BS03}.
\end{proof}

\medskip

\begin{lemma}\label{l:coupling:bridges}
Let $q = \frac1d$. For any $D>0$ there exists $C_{\ref{l:coupling:bridges}}<\infty$ such that for every $n\geq 1$, there exists a coupling $\mathbb Q_n$
of the $d$-dimensional random walk bridge $(X_t)_{t\in[0,2n]}$ of length $2n$ from $0$ to $0$ and the standard $d$-dimensional Brownian bridge $(B_t)_{t\in[0,1]}$ such that 
\[
\mathbb Q_n\left[\sup_{0\leq t\leq 1}\left\|\frac{X_{2nt}}{\sqrt{2nq}} - B_t\right\| > C_{\ref{l:coupling:bridges}}\cdot n^{-\frac14}\log n\right] \leq C_{\ref{l:coupling:bridges}}\cdot n^{-D}.
\]
\end{lemma}
\begin{proof}
The one dimensional case of the lemma follows from \cite[Lemma~3.1]{LT04}. 
(The result proved there is much stronger than the statement of Lemma~\ref{l:coupling:bridges}, 
the exponent $\frac14$ in the event of the lemma is replaced by $\frac12$ in \cite[Lemma~3.1]{LT04}.)

Fix $d\geq 1$ and $n\geq 1$. 
Let $(X^{i,m}, B^{i,m})_{1\leq i\leq d,\, m\geq 1}$ be independent copies of pairs of one-dimensional random walk and Brownian bridges 
each coupled to satisfy the requirements of Lemma~\ref{l:coupling:bridges} in one dimension, where $X^{i,m}$ is distributed as a random walk bridge of length $2m$. 

We will sample the desired $d$-dimensional random walk bridge by first specifying the choice of coordinate directions for all $2n$ steps, and 
then by choosing the directions along specified coordinates. 
Let $L_n = (L^1_n,\ldots, L^d_n)$ be an independent multinomial sequence of length $2n$ with parameter $(q,\ldots,q)$  
conditioned on all $L^1_n(2n),\ldots L^d_n(2n)$ being even. Namely, 
\[
\P[L_n \in \cdot] = \P\left[\left(\sum_{i=1}^kZ_i;\, 1\leq k\leq 2n\right) \in \cdot~\Big|~\sum_{i=1}^{2n}Z_i\in(2\Z)^d\right],
\]
where $Z_i = (Z_{i,1},\ldots,Z_{i,d})$ are independent and $\P[Z_i = e_1] = \ldots = \P[Z_i = e_d] = q$, with $e_1,\ldots, e_d$ being the canonical basis of $\R^d$. 
Let 
\[
M_1 = L^1_n(2n), \, \ldots,\, M_d = L^d_n(2n)
\]
be the number of jumps of each of the coordinates in $L_n$. Consider 
\begin{align*}
X_k &= X^{1,M_1}_{L^1_n(k)} \, e_1 + \ldots + X^{d,M_d}_{L^d_n(k)}\, e_d,\quad 1\leq k\leq 2n,\\[10pt]
B_t &= B^{1,M_1}_t \, e_1 + \ldots + B^{d,M_d}_t\, e_d,\quad 0\leq t\leq 1.
\end{align*}
Then
\begin{itemize}
\item
$(X_t)_{t\in[0,2n]}$ has the law of the $d$-dimensional random walk bridge,
\item
$(B_t)_{t\in[0,1]}$ has the law of the standard $d$-dimensional Brownian bridge. 
\end{itemize}
It remains to show that this coupling satisfies the bound in the statement of the lemma. 
It will be a consequence of the following two claims. 
\begin{claim}\label{cl:multinomial}
For any $D>0$ there exists $C<\infty$ such that 
\[
\P\left[|L^i_n(k) - qk| \leq C(n\log n)^{\frac12},\,\text{for all $1\leq i\leq d$ and $1\leq k\leq 2n$}\right]\geq 1 - Cn^{-D}. 
\]
\end{claim}
\begin{proof}[Proof of Claim~\ref{cl:multinomial}]
First note that 
\[
\P\left[\sum_{i=1}^{2n}Z_i\in(2\Z)^d\right] \geq \P\left[S_{2n}:= S^{1}_{L^1_n(2n)} \, e_1 + \ldots + S^{d}_{L^d_n(2n)}\, e_d = 0\right]
= p_{2n}(0,0) \stackrel{\mathrm{L.}\ref{l:p2n:expansion}}\geq cn^{-\frac d2},
\]
where $S^1,\ldots, S^d$ are independent copies of one dimensional simple random walk on $\Z$ started from the origin (which implies that $S$ is a $d$-dimensional simple random walk).

On the other hand, by \cite[Corollary~12.2.7]{LL10}, 
for any $D>0$ there exists $C<\infty$ such that for each $1\leq i\leq d$, 
\[
\P\left[\max_{1\leq k\leq 2n}\Big|\sum_{j=1}^kZ_{j,i} - qk\Big| > C(n\log n)^{\frac12}\right]\leq Cn^{-D-\frac d2}. 
\]
The two above estimates together and the definition of $L_n$ give the claim. 
\end{proof}

\begin{claim}\label{cl:coupling:delayedrw}
Let $(\ell(k))_{0\leq k\leq 2n}$ be an integer sequence with 
\begin{itemize}\itemsep0pt
\item
$\ell(0) = 0$, 
\item
for all $1\leq k\leq 2n$, $0\leq \ell(k) - \ell(k-1)\leq 1$, 
\item
for some $1\leq C_\ell<\infty$, $\max_{1\leq k\leq 2n}|\ell(k) - qk| \leq C_\ell(n\log n)^{\frac 12}$, 
\item
$m:= \ell(2n)$ is even. 
\end{itemize}
Let $(X^m,B^m)$ be a pair of one-dimensional random walk bridge of length $m$ and a standard one-dimensional Brownian bridge 
coupled to satisfy the requirements of Lemma~\ref{l:coupling:bridges} in one dimension. 
Let 
\[
\widetilde X_k = X^m_{\ell(k)},\quad 0\leq k\leq 2n
\]
be the time reparametrization of the bridge $X^m$ induced by the sequence $\ell$, and 
$(\widetilde X_{2nt})_{t\in[0,1]}$ its linear interpolation. Then
for any $D>0$, there exists $C<\infty$ (depending on $\ell$ only through $C_\ell$) such that 
\[
\P\left[\sup_{0\leq t\leq 1}\left|\frac{\widetilde X_{2nt}}{\sqrt{2nq}} - B^m_t\right| > C \cdot n^{-\frac14}\log n\right] \leq C\cdot n^{-D}.
\]
\end{claim}
\begin{proof}[Proof of Claim~\ref{cl:coupling:delayedrw}]
Let $D>0$. By assumption on $(X^m,B^m)$ and the fact that $|m-2nq| = o(n)$, there exists $C<\infty$ such that 
\[
\P\left[\sup_{0\leq t\leq 1}\left|\frac{X^m_{mt}}{\sqrt{m}} - B^m_t\right| > C \cdot n^{-\frac14}\log n\right] \leq C\cdot n^{-D}.
\]
Thus, it suffices to show that there exists $C<\infty$ such that 
\[
\P\left[\sup_{0\leq t\leq 1}\left|\frac{\widetilde X_{2nt}}{\sqrt{2nq}} - \frac{X^m_{mt}}{\sqrt{m}}\right| > C \cdot n^{-\frac14}\log n\right] \leq C\cdot n^{-D}.
\]
Since $\sqrt m = \sqrt{2nq}\left(1 + O\left((\frac{\log n}{n})^{\frac12}\right)\right)$, there exist $C,C'<\infty$ such that 
\begin{multline*}
\P\left[\sup_{0\leq t\leq 1}\left|\frac{X^m_{mt}}{\sqrt{2nq}} - \frac{X^m_{mt}}{\sqrt{m}}\right| > C \cdot n^{-\frac14}\log n\right] 
\leq 
\P\left[\sup_{0\leq t\leq 1}\left|X^m_{mt}\right| > C' \cdot n^{\frac34}\log^{\frac 12} n\right] \\
= 
\frac{\P\left[\max_{0\leq k\leq m}\left|S_k\right| > C' \cdot n^{\frac34}\log^{\frac 12} n\right]}{p_m(0,0)}
\leq C\cdot n^{-D},
\end{multline*}
where $S$ is a one-dimensional simple random walk on $\Z$. 
The last inequality follows from Lemma~\ref{l:p2n:expansion} and \cite[Corollary~12.2.7]{LL10}. 

It remains to show that there exists $C<\infty$ such that 
\[
\P\left[\sup_{0\leq t\leq 1}\left|\widetilde X_{2nt} - X^m_{mt}\right| > C \cdot n^{\frac14}\log n\right] \leq C\cdot n^{-D}.
\]
By the definition of $\widetilde X$, the probability on the left hand side equals 
\begin{multline*}
\P\left[\sup_{0\leq t\leq 1}\left|X^m_{\ell(2nt)} - X^m_{mt}\right| > C \cdot n^{\frac14}\log n\right]
\leq \frac{2n\cdot \P\left[\max_{0\leq k\leq 4C_\ell (n\log n)^{\frac12}}\left|S_k\right| > C \cdot n^{\frac14}\log n\right]}{p_m(0,0)}\\
\leq C\cdot n^{-D},
\end{multline*}
where $S$ is again a one-dimensional simple random walk on $\Z$. In the first inequality we used the fact that $|\ell(\lfloor 2nt\rfloor) - mt| \leq 3C_\ell (n\log n)^{\frac12}$, 
and in the second the exponential Markov inequality and Lemma~\ref{l:p2n:expansion}.
By putting all the estimates together we get the claim.
\end{proof}

\bigskip

To complete the proof of Lemma~\ref{l:coupling:bridges}, we call the sequence $L_n=(L_n^1,\ldots,L_n^d)$ $C$-good if each one-dimensional sequence $L^i_n$ 
satisfies the assumptions of Claim~\ref{cl:coupling:delayedrw} with $C_\ell = C$. By Claim~\ref{cl:multinomial}, 
for each $D>0$ there exists $C_1<\infty$ such that $\P[\text{$L_n$ is not $C_1$-good}] \leq C_1 \cdot n^{-D}$. 
By Claim~\ref{cl:coupling:delayedrw} applied to each of the $d$ projections of $X$ and $B$, there exists $C<\infty$ such that 
\[
\P \left[\sup_{0\leq t\leq 1}\left\|\frac{X_{2nt}}{\sqrt{2nq}} - B_t\right\| > C\cdot n^{-\frac14}\log n~\Big|~\text{$L_n$ is $C_1$-good}\right] \leq C n^{-D}.
\]
The two estimates together give the result of the lemma. 
\end{proof}

\subsection{Proof of Theorem~\ref{thm:coupling:soups}}

The main strategy of the proof is similar to that of \cite{LT04}. We will pair-up random walk loops of length $2n\geq N^\theta$ rooted at $z\in [-rN,rN]^d$
and Brownian loops of length between $\frac{2(n-1)}{d} + \blrange$ and $\frac{2n}{d} + \blrange$ and rooted in $z + [-\frac12,\frac12]^d$. 

Fix $\lambda>0$, $r>0$, and $N\geq 1$. For $n\geq 1$, let   
\[
\widetilde q_n = \frac{1}{2n}\cdot p_{2n}(0,0) \stackrel{\mathrm{L.}\ref{l:p2n:expansion}}= 
\frac1n\, \left(\frac{d}{4\pi n}\right)^{\frac d2}\cdot \left(1 - \frac{d}{8n} + O\left(\frac{1}{n^2}\right)\right), 
\]
and 
\[
q_n = \int_{\frac{2(n-1)}{d} + \blrange}^{\frac{2n}{d} + \blrange}\, \frac{ds}{s\cdot (2\pi s)^{\frac d2}} = 
\frac1n\, \left(\frac{d}{4\pi n}\right)^{\frac d2}\cdot \left(1 - \frac{d}{8n} + O\left(\frac{1}{n^2}\right)\right). 
\]
The main reason to define $\blrange$ as in \eqref{def:blrange} is so that the first two terms in the expansions of $\widetilde q_n$ and $q_n$ coincide. 
In particular, 
\begin{equation}\label{eq:qn:difference}
\left|q_n - \widetilde q_n\right| \leq C\cdot n^{-\frac d2 - 3}.
\end{equation}
Let 
\[
\widetilde N(z,n;\cdot),\quad n\in\{1,2,\ldots\}, z\in\Z^d,
\]
be independent Poisson point processes on $(0,\infty)$ with intensity parameter $\widetilde q_n$, and 
\[
N(z,n;\cdot),\quad n\in\{1,2,\ldots\}, z\in\Z^d,
\]
independent Poisson point processes on $(0,\infty)$ with intensity parameter $q_n$.
We couple all these processes so that the pairs of processes 
\[
\left\{\widetilde N(z,n;\cdot), N(z,n;\cdot)\right\},\quad n\in\{1,2,\ldots\}, z\in\Z^d,
\]
are independent and 
\[
\P\left[\widetilde N(z,n;t)\neq N(z,n;t)\right]\leq C\cdot t\cdot n^{-\frac d2 - 3}.
\]
Let 
\[
\widetilde L(z,n;m),\quad n\in\{1,2,\ldots\}, z\in\Z^d, m\in \{1,2,\ldots\},
\]
be independent random walk bridges of length $2n$ rooted at $0$, and 
\[
L(z,n;m),\quad n\in\{1,2,\ldots\}, z\in\Z^d, m\in \{1,2,\ldots\},
\]
independent standard Brownian bridges coupled with the respective $\widetilde L(z,n;m)$ as in Lemma~\ref{l:coupling:bridges}. 
All the $\widetilde L$'s and $L$'s are assumed to be independent from all the $\widetilde N$'s and $N$'s. 

The random walk loop soup $\rwls^\lambda$ of intensity $\lambda$ is defined as in \eqref{def:rwls:2} as the collection (multiset) of loops
\[
\left\{z+\widetilde L(z,n;m)~:~n\in\{1,2,\ldots, \},\, z\in\Z^d,\, m\in \left\{1,\ldots, \widetilde N(z,n;\lambda)\right\}\right\} .
\]

In order to define the Brownian loop soup, we introduce more random variables. Let 
\[
Y(z,n;m),\quad n\in\{1,2,\ldots\}, z\in\Z^d, m\in \{1,2,\ldots\},
\]
be independent random variables uniformly distributed on $[-\frac12,\frac12]^d$, and 
\[
T(z,n;m),\quad n\in\{1,2,\ldots\}, z\in\Z^d, m\in \{1,2,\ldots\},
\]
independent real-valued random variables with density 
\[
\propto\, \frac{1}{t\cdot (2\pi t)^{\frac d2}},\qquad\frac{2(n-1)}{d} + \blrange\leq t\leq \frac{2n}{d} + \blrange.
\]
We assume that the $Y$'s and $T$'s are independent and also independent of all the previously defined variables. 

Finally, consider an independent Brownian loop soup $\bls_\blrange$ of loops of time duration $\leq \blrange$, 
and the corresponding restriction of $\bls_\blrange$ to $\blspace\times(0,\lambda)$, $\bls_\blrange^\lambda$ (the loops ``appearing before time $\lambda$''). 

To generate the Brownian loop soup, we first rescale each loop $L(z,n;m)$ so that its time duration is $T(z,n;m)$ and 
translate so that its root is at $z+ Y(z,n;m)$, obtaining the loop $L^*(z,n;m)$. The Brownian loop soup $\bls^\lambda$ of intensity $\lambda$ is then 
the collection of loops 
\[
\left\{L^*(z,n;m)~:~n\in\{1,2,\ldots, \},\, z\in\Z^d,\, m\in \left\{1,\ldots, N(z,n;\lambda)\right\}\right\}\cup\bls_\blrange^\lambda .
\]

\bigskip

It remains to show that the constructed coupling satisfies all the requirements of the theorem. 
We would like to pair-up random walk loops from the multiset 
\[
\left\{z+\widetilde L(z,n;m)~:~ m\in \left\{1,\ldots, \widetilde N(z,n;\lambda)\right\}\right\}
\]
with Brownian loops from the set 
\[
\left\{L^*(z,n;m)~:~ m\in \left\{1,\ldots, N(z,n;\lambda)\right\}\right\},
\]
for all $n$ and $z$ such that $2n\geq N^\theta$ and $|z|\leq rN$. 

First, we estimate the probability that cardinalities are different, 
\begin{multline*}
\P\left[N(z,n;\lambda)\neq \widetilde N(z,n;\lambda)~\text{for some}~n\geq \frac12N^\theta, |z|\leq rN \right]
\stackrel{\eqref{eq:qn:difference}}\leq C(rN)^d\,\lambda\sum_{2n\geq N^\theta} n^{-\frac d2 - 3} \\
\leq C\,\lambda\,r^d\,N^{d-\theta(\frac d2 + 2)}.
\end{multline*}

Next, we rule out the existence of big loops, 
\[
\P\left[\widetilde N(z,n;\lambda)>0~\text{for some}~n\geq N^3, |z|\leq rN \right]
\leq C(rN)^d\,\lambda\sum_{n\geq N^3} n^{-\frac d2 - 1} \\
\leq C\,\lambda\,r^d\,N^{-\frac d2}.
\]

Now, we bound the probability of having too many loops of a given size rooted at the same vertex, 
\begin{multline*}
\P\left[\widetilde N(z,n;\lambda)\geq N^5~\text{for some}~\frac12N^\theta\leq n\leq N^3, |z|\leq rN \right]
\leq C(rN)^d\,N^3\,\frac{\lambda N^{-\theta(\frac d2 + 1)}}{N^5}\\
\leq C\,\lambda\,r^d\,N^{d-\theta(\frac d2 + 2)},
\end{multline*}
where in the first inequality we used the Markov inequality and the fact that $\widetilde N(z,n;\lambda)$ is 
a Poisson random variable with parameter $\lambda \widetilde q_n$. 

\medskip

By putting all the bounds together and using Lemma~\ref{l:coupling:bridges}, 
\begin{multline*}
\P\left[
\begin{array}{c}\text{there exist loops $L^*(z,n;m)$ and $z+ \widetilde L(z,n;m)$}\\ 
\text{for some $n\geq \frac12 N^\theta$, $|z|\leq rN$, $m\leq \widetilde N(z,n;\lambda)$}\\
\text{such that sup-distance between them is $\geq C_{\ref{l:coupling:bridges}}N^{\frac 34}\log N$}
\end{array}
\right]\\
\leq C\, \lambda\, r^d\, N^{-\min\left(\frac{d}{2},\,\theta(\frac d2+2)-d\right)}
+ C (rN)^d\cdot N^3\cdot N^5\cdot N^{-\theta D} \\
\leq 
C\, (\lambda+1)\, r^d\, N^{-\min\left(\frac{d}{2},\,\theta(\frac d2+2)-d\right)},
\end{multline*}
by choosing $D$ sufficiently large. The proof is complete. \qed

\section{LEW scaling limit + Brownian loop soup = Brownian motion}\label{sec:BMdecompositionproof}

In this section we prove Theorem~\ref{thm:K+BLS=BM} using the Beurling-type estimate of Theorem~\ref{thm:beurling} 
and the strong coupling of loop soups from Theorem~\ref{thm:coupling:soups}. 
We will focus here, without further mentioning, on the three dimensional case. 
The two dimensional case can be treated similarly and often with less effort, so we leave the details to an interested reader. 

\subsection{Preliminaries}\label{sec:BMdecompositionproof:preliminaries}

Recall the notation for random walk loop soups from Section~\ref{sec:rwls}. 
Let $U$ be a connected open subset of $\R^3$ with smooth boundary. 
Let $0<\varepsilon<\delta$ be sufficiently small so that all the finite connected components of $U^c$ have diameter $>\delta$. 
For $n\geq 1$, let $U_n = nU$ and use the same notation for $nU\cap\Z^3$. Let $\lambda>0$. 
The first result states that it is unlikely that there is a loop in $\rwls^\lambda$ of big diameter which is  
contained in $U_n$ and reaches very close to the boundary of $U_n$. 
\begin{proposition}\label{prop:loopsatboundary}
There exists $\alpha>0$ and $C = C(\alpha,U)<\infty$ such that for all $\lambda>0$, $0<\varepsilon<\delta$ as above, 
and $n\geq 1$, 
\begin{equation}\label{eq:loopsatboundary}
\P\left[
\begin{array}{c}
\text{There exists $\ell\in\rwls^\lambda$ with diameter $>\delta n$}\\
\text{such that $\ell\subset U_n$ and $\ell\nsubseteq U_n^{-\varepsilon n}$}
\end{array}
\right]\leq C\,\lambda\,\frac{\varepsilon^{\alpha}}{\delta^5}.
\end{equation}
\end{proposition}
\begin{proof}
Let $\varepsilon<\gamma<\delta$. We will prove that for some $\eta>0$ and $C=C(U,\eta)<\infty$, 
\begin{equation}\label{eq:loopsatboundary:gamma}
\rwlm\left[\ell~:~\text{diameter of $\ell$ is $>\delta n$, $\ell\subset U_n$, and $\ell\nsubseteq U_n^{-\varepsilon n}$}\right]
\leq C\cdot \left(\frac{\gamma^\eta}{\delta^{\eta+3}} + \frac{\varepsilon^\eta}{\gamma^{\eta+3}}\right).
\end{equation}
Optimizing over $\gamma$ then gives that the above measure is $\leq C\,\varepsilon^\alpha\,\delta^{-5}$ for $\alpha = \frac{\eta^2}{2\eta + 3}$. 
Since the probability in \eqref{eq:loopsatboundary} equals 
\[
1 - \exp\left(-\lambda\,\rwlm\left[\ell~:~\text{diameter of $\ell$ is $>\delta n$, $\ell\subset U_n$, and $\ell\nsubseteq U_n^{-\varepsilon n}$}\right]\right),
\]
the result follows. 
Thus, it suffices to prove \eqref{eq:loopsatboundary:gamma}. 
Denote by $L_n$ the set of all loops $\ell$ with diameter $>\delta n$, $\ell\subset U_n$, and $\ell\nsubseteq U_n^{-\varepsilon n}$. By the definition of $\rwlm$, 
\begin{eqnarray*}
\rwlm[L_n] &= &\sum_{z\in \Z^3\cap U_n}\,\sum_{k\geq 1} \frac{1}{2k}\,\P^{z}\left[R(2k) = z,~R[0,2k]\in L_n\right]\\
&\leq &C(U)n^3\cdot \max_{z\in \Z^3\cap U_n}\,\sum_{k\geq 1} \frac{1}{2k}\,\P^{z}\left[R(2k) = z,~R[0,2k]\in L_n\right].
\end{eqnarray*}
Let $z\in \Z^3\cap U_n$. We consider separately two cases: $k\leq \gamma^2 n^2$ and $k\geq \gamma^2 n^2$. 

If $k\leq \gamma^2 n^2$, then by the time reversibility of the random walk, 
\begin{equation}\label{eq:loopsatboundary:1}
\P^{z}\left[R(2k) = z,~R[0,2k]\in L_n\right]
\leq 2\,\P^z\left[T_{z,\frac12\delta n} \leq k,~R(2k)=z\right].
\end{equation}
By the local central limit theorem, for each $y\in\Z^d$ and $m$, 
\[
\P^y\left[R(m) = z\right]\leq C\,m^{-\frac32}\,e^{-\frac{3|y-z|^2}{2m}}.
\]
Using the strong Markov property on exiting from $B(z,\frac12\delta n)$ and the local central limit theorem, one gets from \eqref{eq:loopsatboundary:1} that 
\[
\P^{z}\left[R(2k) = z,~R[0,2k]\in L_n\right]
\leq C\,k^{-\frac32} e^{-\frac{3(\delta n)^2}{16k}} 
\leq C\,e^{-\frac{3}{32}\,\left(\frac{\delta}{\gamma}\right)^2}\,k^{-\frac32} e^{-\frac{3(\delta n)^2}{32k}}.
\]
Thus, 
\[
n^3\,\sum_{k=1}^{\lfloor \gamma^2 n^2\rfloor} \frac{1}{2k}\,\P^{z}\left[R(2k) = z,~R[0,2k]\in L_n\right]
\leq C\,e^{-\frac{3}{32}\,\left(\frac{\delta}{\gamma}\right)^2}\,n^3\,\sum_{k=1}^\infty k^{-\frac52} e^{-\frac{3(\delta n)^2}{32k}}
\leq C'\,e^{-\frac{3}{32}\,\left(\frac{\delta}{\gamma}\right)^2}\,\frac{1}{\delta^3}. 
\]
Note that for any $\eta>0$ there exists $C = C(\eta)<\infty$ such that $e^{-\frac{3}{32}\,\left(\frac{\delta}{\gamma}\right)^2}\leq C\,\left(\frac{\gamma}{\delta}\right)^\eta$. 
Thus, we have the first part of the bound in \eqref{eq:loopsatboundary:gamma}.

Next we consider $k\geq \gamma^2 n^2$. By the time reversibility of the random walk, 
\begin{equation}\label{eq:loopsatboundary:2}
\P^{z}\left[R(2k) = z,~R[0,2k]\in L_n\right] \leq 
2\,\P^z\left[T\left( \left(U_n^{-\varepsilon n}\right)^c\right) \leq k,~T(\partial U_n) > \frac32k,~R(2k) = z\right],
\end{equation}
It is well known (see, e.g., \cite[Lemma~2.5]{K}) that there exist $\eta>0$ and $C=C(U)<\infty$ such that 
for all $y\in U_n\setminus U_n^{-\varepsilon n}$,  
\[
\P^y\left[T(\partial U_n)>\frac12\gamma^2 n^2\right] \leq C\left(\frac{\varepsilon}{\gamma}\right)^\eta .
\]
By the strong Markov property at time $T\left(\left( U_n^{-\varepsilon n}\right)^c\right)$ and the above inequality, as well as 
the Markov property at time $\frac32k$ and the local central limit theorem, 
one gets from \eqref{eq:loopsatboundary:2} that 
\[
\P^{z}\left[R(2k) = z,~R[0,2k]\in L_n\right]
\leq C\,\left(\frac{\varepsilon}{\gamma}\right)^\eta\,k^{-\frac32} .
\]
Thus, 
\[
n^3\,\sum_{k\geq \lfloor \gamma^2 n^2\rfloor} \frac{1}{2k}\,\P^{z}\left[R(2k) = z,~R[0,2k]\in L_n\right]
\leq C\,\left(\frac{\varepsilon}{\gamma}\right)^\eta\,n^3\,\sum_{k\geq \lfloor \gamma^2 n^2\rfloor} k^{-\frac52}
\leq C'\,\left(\frac{\varepsilon}{\gamma}\right)^\eta\,\frac{1}{\gamma^3}. 
\]
This gives us the second part of the bound in \eqref{eq:loopsatboundary:gamma}. 
The proof of Proposition~\ref{prop:loopsatboundary} is complete. 
\end{proof}

\medskip

Consider a simple random walk $R^1$ from $0$ and an independent random walk loop soup $\rwls^\lambda$ defined on the same probability space. 
\begin{proposition}\label{prop:loopsatlew}
There exists $\alpha>0$ and $C = C(\alpha)<\infty$ such that for all $\lambda>0$, $0<\varepsilon<\delta$, and $n\geq 1$, 
\begin{equation}\label{eq:loopsatlew}
\P\left[
\begin{array}{c}
\text{There exists $\ell\in\rwls^\lambda$ with diameter $>\delta n$ such that $\ell\subset B(0,n)$,}\\
\text{$\mathrm{dist}(\ell,\loe(R^1[0,T^1_{0,n}]))<\varepsilon n$ and $\ell\cap \loe(R^1[0,T^1_{0,n}])=\emptyset$}
\end{array}
\right]\leq C\,\left(\sqrt{\varepsilon} + \lambda\,\frac{\varepsilon^{\alpha}}{\delta^5}\right).
\end{equation}
\end{proposition}
\begin{proof}
Let $\eta>0$ and define the event 
\[
A = \left\{
\begin{array}{c}
\text{For all $x\in B(0,n)$ with $\mathrm{dist}(x,\loe(R^1[0,T^1_{0,n}]))\leq \varepsilon n$,}\\
\P^{2,x}\left[R^2[0,T^2_{x,\varepsilon^{\frac14}n}]\cap\loe(R^1[0,T^1_{0,n}])=\emptyset\right]<\varepsilon^\eta
\end{array}
\right\}.
\]
By Theorem~\ref{thm:beurling}, there exists $\eta>0$ and $C=C(\eta)<\infty$ such that $\P^{1,0}[A^c]\leq C\,\sqrt{\varepsilon}$. 
This gives the first part of the bound in \eqref{eq:loopsatlew}, and it remains to estimate the probability of the event in \eqref{eq:loopsatlew} intersected with $A$. 
By the definition of Poisson point process and independence of $R^1$ and $\rwls^\lambda$, this probability equals 
\begin{multline*}
\E^{1,0}\left[\mathds{1}_A\cdot \left(1 - \exp\left(-\lambda\rwlm\left[
\ell~:~ 
\begin{array}{c}
\text{diameter of $\ell$ is $>\delta n$, $\ell\subset B(0,n)$,}\\ 
\mathrm{dist}(\ell,\loe(R^1[0,T^1_{0,n}]))<\varepsilon n,\\
\text{and $\ell\cap \loe(R^1[0,T^1_{0,n}])=\emptyset$}
\end{array}
\right]\right)\right)\right]\\
\leq 
\E^{1,0}\left[\mathds{1}_A\cdot \lambda\rwlm\left[
\ell~:~ 
\begin{array}{c}
\text{diameter of $\ell$ is $>\delta n$, $\ell\subset B(0,n)$,}\\ 
\mathrm{dist}(\ell,\loe(R^1[0,T^1_{0,n}]))<\varepsilon n,\\
\text{and $\ell\cap \loe(R^1[0,T^1_{0,n}])=\emptyset$}
\end{array}
\right]\right]\\
=:
\E^{1,0}\left[\mathds{1}_A\cdot \lambda\rwlm[L_n]\right],
\end{multline*}
where we denoted by $L_n$ the set of loops $\ell\subset B(0,n)$ with diameter $>\delta n$ and such that 
$\mathrm{dist}(\ell,\loe(R^1[0,T^1_{0,n}]))<\varepsilon n$ and $\ell\cap \loe(R^1[0,T^1_{0,n}])=\emptyset$.
It suffices to prove that for some $\alpha>0$ and $C=C(\alpha)<\infty$, 
\[
\mathds{1}_A\cdot\rwlm[L_n]\leq  C\,\frac{\varepsilon^{\alpha}}{\delta^5}.
\]
By the definition of $\rwlm$, 
\begin{multline*}
\rwlm[L_n]
=
\sum_{z\in \Z^3\cap B(0,n)}\,\sum_{k\geq 1} \frac{1}{2k}\,
\P^{2,z}\left[R^2(2k) = z,~ R^2[0,2k]\in L_n\right]\\
\leq 
C\,n^3\,\max_{z\in \Z^3\cap B(0,n)}\,\sum_{k\geq 1} \frac{1}{2k}\,\P^{2,z}\left[R^2(2k) = z,~ R^2[0,2k]\in L_n\right].
\end{multline*}
Let $\varepsilon<\gamma<\delta$ be some paramter to be fixed later. 
We first consider the case $k\leq \gamma^2n^2$. 
Exactly as in the proof of Proposition~\ref{prop:loopsatboundary} (see below \eqref{eq:loopsatboundary:1}), there exists $C<\infty$ such that 
\[
n^3\,\sum_{k=1}^{\lfloor \gamma^2n^2\rfloor} \frac{1}{2k}\,
\P^{2,z}\left[R^2(2k) = z,~ R^2[0,2k]\in L_n\right]\leq C\,\frac{\gamma}{\delta}\cdot\frac{1}{\delta^3}.
\]
We consider next the case $k\geq \gamma^2n^2$. 
Define $S = \loe(R^1[0,T^1_{0,n}])$. 
By the time reversibility of the random walk, 
\begin{multline*}
\P^{2,z}\left[R^2(2k) = z,~R^2[0,2k]\in L_n\right]\\ 
\leq 
2\,\P^{2,z}\left[R^2[0,2k]\subset B(0,n),~\mathrm{dist}(R^2[0,k],S)\leq \varepsilon n,~R^2[0,\frac32k]\cap S=\emptyset,~R^2(2k) = z\right].
\end{multline*}
Let $T^*=\min\{t:\mathrm{dist}(R^2[0,t],S)\leq \varepsilon n\}$ and $T^{**} = \inf\{t>T^*:R^2(t)\notin B(R^2(T^*),\varepsilon^{\frac14}n)\}$. 
By the Markov inequality, 
\[
\P^{2,z}\left[T^*<\infty,~T^{**}>T^* + \frac12\gamma^2n^2\right]\leq C \frac{\sqrt{\varepsilon}}{\gamma^2}, 
\]
and by the definition of the event $A$, 
\[
\mathds{1}_A\cdot \P^z\left[T^*<\infty,~R^2(T^*)\in B(0,n),~R^2[T^*,T^{**}]\cap S=\emptyset\right]
\leq \varepsilon^{\eta}. 
\]
Combining these bounds with the Markov property at time $\frac32k$ and the local central limit theorem, we obtain that 
for each $k\geq \gamma^2 n^2$, 
\[
\mathds{1}_A\cdot \P^{2,z}\left[R^2(2k) = z,~ R^2[0,2k]\in L_n\right]\leq 
C\,\left(\frac{\sqrt{\varepsilon}}{\gamma^2} + \varepsilon^\eta\right)\, k^{-\frac32}. 
\]
Thus, 
\begin{multline*}
\mathds{1}_A\cdot n^3\,\sum_{k\geq \lfloor \gamma^2n^2\rfloor} \frac{1}{2k}\,
\P^{2,z}\left[R^2(2k) = z,~ R^2[0,2k]\in L_n\right]\\
\leq
C\,\left(\frac{\sqrt{\varepsilon}}{\gamma^2} + \varepsilon^\eta\right)\,n^3\,\sum_{k\geq \lfloor \gamma^2n^2\rfloor} k^{-\frac52}
\leq C'\,\left(\frac{\sqrt{\varepsilon}}{\gamma^2} + \varepsilon^\eta\right)\,\frac{1}{\gamma^3}.
\end{multline*}
Putting the two cases together, we get 
\[
\mathds{1}_A\cdot\rwlm[L_n]\leq
C\,\left(\frac{\gamma}{\delta^4} + \left(\frac{\sqrt{\varepsilon}}{\gamma^2} + \varepsilon^\eta\right)\,\frac{1}{\gamma^3}\right).
\]
If $\gamma = \varepsilon^{\frac{\eta}{4}}\delta^{1-\frac{\eta}{4}}$, then the last expression is $\leq C'\,\frac{\varepsilon^\alpha}{\delta^5}$ for $\alpha = \frac{\eta}{4}$. 
This gives the second part of the bound in \eqref{eq:loopsatlew}. 
The proof of Proposition~\ref{prop:loopsatlew} is complete. 
\end{proof}

\subsection{Proof of Theorem~\ref{thm:K+BLS=BM}}\label{sec:BMdecompositionproof:proof}

Let $\bls_D$ be the restriction of an independent Brownian loop soup of intensity $1$ to the loops entirely contained in $D$. 
Denote by $X$ the random subset of $\overline D$ consisting of $\koz$ and all the loops from $\bls_D$ that intersect $\koz$. 
First of all, note that $X$ is closed. Indeed, for any $\varepsilon>0$, the set of loops in $\bls_D$ with diameter bigger than $\varepsilon$ 
is almost surely finite. Thus, the complement of any $\varepsilon$-neighborhood of $\koz$ in $X$ is closed. 
Since $\koz$ is closed, $X$ is also closed. 

Let $\bm$ be the trace of the Brownian motion killed on exiting from $D$ viewed as a compact subset of $\overline D$. 
We need to prove that $X$ has the same law as $\bm$ in $(\mathcal K_D,d_H)$. 
Let $(\Omega,\mathcal F,\P)$ be a probability space large enough to contain all the random variables used in this proof. 
It suffices to prove that 
\[
\P[X\subset U\cap \overline D] = \P[\bm\subset U\cap \overline D],\qquad\text{for all open $U\subset \R^3$.}
\]
Since every bounded open set can be approximated from inside by a finite union of open balls, 
which itself can be approximated from inside by an open set with smooth boundary, 
it suffices to prove the above identity only for sets $U$ with smooth boundary. 

\medskip

For a (measurable) set $S\subset \R^3$ and a (countable) collection of loops $L$ in $\R^3$, let 
$E(S,L)$ be the union of $S$ and all the loops from $L$ that intersect $S$, the {\it enlargement} of $S$ by $L$. 
In particular, $X = E(\koz,\bls_D)$. Also for $\delta>0$, let $L^{>\delta}$ be the subcollection of all the loops from $L$ 
with diameter $>\delta$.

Let $\ls_n$ be the restriction of the random walk loop soup of intensity $1$ rescaled by $\frac1n$ to $D$, i.e., 
$\ls_n = \{\frac1n\ell~:~\ell\in\rwls^1,~\frac1n\ell\subset D\}$. 
By Proposition~\ref{prop:RWdecomposition}, if $\lew_n$ and $\ls_n$ are independent, then $E(\lew_n,\ls_n)$ has the same law as the trace of a simple random walk on $\frac1n\Z^3$
killed on exiting from $D$. 
In particular, as $n\to\infty$, $E(\lew_n,\ls_n)$ converges weakly in the space $(\mathcal K_D,d_H)$ to the Brownian motion $\bm$. 

\medskip

Let $U$ be an open subset of $\R^3$ with a smooth boundary and such that $0\in U$. 
Fix $0<\varepsilon<\delta$.
We now define the random objects that will be used in the proof. 

Since the space $(\mathcal K_D,d_H)$ is separable and $\lew_n$ converges weakly to $\koz$, 
by Skorokhod's representation theorem we can define $(\lew_n)_{n\geq 1}$ and $\koz$ on the same probability space 
so that $d_H(\lew_n,\koz)\to 0$ almost surely. Consider the event
\[
A_0=\left\{\koz\subseteq U^{-3\delta}\right\}\cup\{K\nsubseteq U\}
\]
that if $\koz$ is in $U$ then the distance from $\koz$ to the complement of $U$ is $>3\delta$. By monotonicity, $\P[A_0]\to 1$ as $\delta\to0$. 
For each $n\geq 1$, consider also the event
\[
A_{1,n} = \left\{d_H(\lew_n,\koz)<\varepsilon\right\}.
\]
By construction, $\P[A_{1,n}]\to 1$ as $n\to\infty$.

For each $n\geq 1$, let $(\rwls^1_n,\bls^1_n)$ be independent pairs of the rescaled random walk loop soup of intensity $1$ on $\frac1n\Z^3$ and 
the Brownian loop soup of intensity $1$ on $\R^3$ coupled so that on an event $A_{2,n}$ of probability $>1-Cn^{-\frac12}$ there is a one-to-one correspondence between the loops from $\rwls^1_n$ 
of diameter $>\delta$ rooted in $D_2$ and those from $\bls^1_n$ rooted in $D_2$ and so that the Hausdorff distance between 
the paired loops is $<\varepsilon$. This coupling is possible by Theorem~\ref{thm:coupling:soups}. (In Theorem~\ref{thm:coupling:soups} we paired loops of sufficiently large length, but 
each loop of length $s$ is of diameter of order $\sqrt{s}$.) 
We also assume that all the pairs $(\rwls^1_n,\bls^1_n)$ are independent from $(\lew_n)_{n\geq 1}$ and $\koz$. 

In addition, for each $n\geq 1$, we consider the event $A_{3,n}$ that every loop from $\rwls^1_n$ with diameter $>\delta$ which is contained in $(U\cap D)^{+4\varepsilon}$ 
is also contained in $(U\cap D)^{-4\varepsilon}$, and the event $A_{4,n}$ that every loop from $\rwls^1_n$ with diameter $>\delta$ at distance $<4\varepsilon$ from $\lew_n$ 
intersects $\lew_n$. 
In other words, in the event $A_{3,n}$ there are no big loops that are contained in $(U\cap D)^{+4\varepsilon}$ and get too close to the boundary of $U\cap D$, 
and in the event $A_{4,n}$ there are no big loops that get too close to the $\lew_n$ without hitting it. 
It is proved in Propositions~\ref{prop:loopsatboundary} and \ref{prop:loopsatlew} that for some $\alpha>0$, $\inf_n\P[A_{3,n}\cap A_{4,n}]\geq 1 - C\,\frac{\varepsilon^\alpha}{\delta^5}$. 

Note that in the event $A_{2,n}\cap A_{3,n}$, for every loop from $\rwls^1_n$ with diameter $>\delta$ contained in $D$, its pair from $\bls^1_n$ is contained in $D$, and vice versa. 

We call the restriction of $\rwls^1_n$ to the loops contained in $D$ by $\ls_n$, and 
the restriction of $\bls^1_n$ to the loops contained in $D$ by $\bs_n$. 

\medskip

We prove that for any $n\geq 1$, on the event $A_n = A_0\cap A_{1,n}\cap A_{2,n}\cap A_{3,n}\cap A_{4,n}$, 
\begin{equation}\label{eq:lew:K:inclusion}
\left\{E(\lew_n,\ls_n)\subset U^{-3\varepsilon}\cap\overline D\right\} \subseteq \left\{E(\koz^{+2\varepsilon},\bs_n)\subset U\cap\overline D\right\} \subseteq 
\left\{E(\lew_n,\ls_n)\subset U^{+\delta}\cap\overline D\right\}.
\end{equation}
We begin with a proof of the first inclusion. Assume that $E(\lew_n,\ls_n)\subset U^{-3\varepsilon}$. In particular, $\lew_n\subset U^{-3\varepsilon}$. 
Since $A_{1,n}$ holds, this implies that $\koz\subset U^{-2\varepsilon}$, which is equivalent to $\koz^{+2\varepsilon}\subset U$. 

Let $\ell\in \bs_n^{>\delta}$ be such that $\ell\cap \koz^{+2\varepsilon}\neq\emptyset$. 
Then, since $A_{1,n}$ occurs, $\ell\cap \lew_n^{+3\varepsilon}\neq\emptyset$. 
Since $A_{2,n}\cap A_{3,n}$ occurs, there is $\widetilde \ell\in\ls_n^{>\delta}$ such that $d_H(\widetilde \ell,\ell)<\varepsilon$. 
In particular, $\widetilde \ell\cap\lew_n^{+4\varepsilon}\neq\emptyset$. Since $A_{4,n}$ occurs, this implies that $\widetilde \ell\cap \lew_n\neq \emptyset$. 
By our assumption, $\widetilde \ell \subset U^{-3\varepsilon}$. Therefore, $\ell\subset U^{-2\varepsilon}\subset U$. 
Thus, any $\ell\in \bs_n^{>\delta}$ such that $\ell\cap \koz^{+2\varepsilon}\neq\emptyset$ is contained in $U$. 
This implies that $E(\koz^{+2\varepsilon},\bs_n^{>\delta})\subset U$. 
Since $A_0$ occurs, the above is true if and only if $E(\koz^{+2\varepsilon},\bs_n)\subset U$.

We proceed with the proof of the second inclusion in \eqref{eq:lew:K:inclusion}. 
Assume that $E(\koz^{+2\varepsilon},\bs_n)\subset U$. Since $A_0$ occurs, this is equivalent to $E(\koz^{+2\varepsilon},\bs_n^{>\delta})\subset U$. 
Since $A_{1,n}$ occurs and $\koz^{+2\varepsilon}\subset U$, we also have $\lew_n\subset U$. 

Let $\widetilde \ell\in\ls_n^{>\delta}$ such that $\widetilde \ell\cap\lew_n\neq\emptyset$. 
Since $A_{1,n}$ occurs, $\widetilde \ell\cap \koz^{+\varepsilon} \neq \emptyset$. 
Since $A_{2,n}\cap A_{3,n}$ occurs, there is $\ell\in\bs_n^{>\delta}$ such that $d_H(\widetilde \ell,\ell)<\varepsilon$. 
In particular, $\ell\cap \koz^{+2\varepsilon}\neq\emptyset$. 
By assumption, $\ell\subset U$. Therefore, $\widetilde \ell\subset U^{+\varepsilon}$. 
Since $A_{3,n}$ occurs, we actually have $\widetilde \ell\subset U$. 
Thus, any $\widetilde \ell\in \ls_n^{>\delta}$ such that $\widetilde \ell\cap \lew_n \neq\emptyset$ is contained in $U$. 
This implies that $E(\lew_n,\ls_n^{>\delta})\subset U$. 
Finally, by adding all the loops of diameter $<\delta$ we get $E(\lew_n,\ls_n)\subset U^{+\delta}$. 
The proof of the inclusion \eqref{eq:lew:K:inclusion} is complete. 

\bigskip

It follows from \eqref{eq:lew:K:inclusion} that for all $n\geq 1$, $0<\varepsilon<\delta$, 
\begin{multline*}
\P\left[E(\lew_n,\ls_n)\subset U^{-3\delta}\cap\overline D\right] - \P[A_n^c] \leq 
\P\left[E(\koz^{+2\varepsilon},\bls_D)\subset U\cap\overline D\right] \\
\leq \P\left[E(\lew_n,\ls_n)\subset U^{+\delta}\cap\overline D\right] + \P[A_n^c].
\end{multline*}
By monotonicity, 
\[
\lim_{\varepsilon\to 0}\P\left[E(\koz^{+2\varepsilon},\bls_D)\subset U\cap\overline D\right] = 
\P\left[E(\koz,\bls_D)\subset U\cap\overline D\right] = \P[X\subset U\cap\overline D]. 
\]
Since 
$\limsup_{\delta\to 0}\limsup_{\varepsilon\to 0}\limsup_{n\to\infty}\P[A_n^c] = 0$,
we have
\[
\liminf_{\delta\to 0}\P\left[\bm\subset U^{-3\delta}\cap\overline D\right]\leq  \P\left[X\subset U\cap\overline D\right] 
\leq \limsup_{\delta\to 0}\P\left[\bm\subset U^{+\delta}\cap\overline D\right]. 
\]
Since by monotonicity both left and right hand sides are equal to $\P\left[\bm\subset U\cap\overline D\right]$, 
we get the desired result. 
\qed

\section{Scaling limit is a simple path}\label{sec:simplepath}

\subsection{Quasi-loops}\label{sec:quasiloops}

We say that a nearest neighbor path $\gamma$ in $\Z^3$ has a $(s,r)$-{\it quasi-loop} at $v\in\Z^3$ 
if there exist $i\leq j$ such that $\gamma(i),\gamma(j)\in B(v,s)$ and $\gamma[i,j]\nsubseteq B(v,r)$. 
The set of all such $v$'s is denoted by $\ql(s,r;\gamma)$. 
Note that the set $\ql(s,r; \gamma)$ is non-increasing in $r$ and non-decreasing in $s$. 

\begin{theorem}\label{thm:ql}
There exist $M<\infty$, $\alpha>0$, and $C<\infty$ such that for any $\varepsilon>0$ and $n\geq 1$, 
\begin{equation}\label{eq:ql}
\P^0\left[
\ql\left(\varepsilon^Mn,\sqrt{\varepsilon}n;~\loe(R[0,T_{0,n}])\right)\neq\emptyset
\right]\leq C\varepsilon^\alpha.
\end{equation}
\end{theorem}
\begin{proof}
It suffices to consider sufficiently small $\varepsilon$. Let $M\geq 1$. 
The proof is subdivided into three cases: (1) there is a quasi-loop at a vertex of $B(0,\varepsilon n)$, 
(2) there is a quasi-loop at a vertex of $B(0,n-\varepsilon n)^c$, 
and (3) there is a quasi-loop at a vertex of $B(0,n-\varepsilon n)\setminus\overline{B(0,\varepsilon n)}$. 
We will verify the first and the second cases for $M=1$ and the last case for a sufficiently large $M$. 

Consider the first case. If $\ql\left(\varepsilon n,\sqrt{\varepsilon}n;~\loe(R[0,T_{0,n}])\right)\cap B(0,\varepsilon n)\neq\emptyset$, 
then the random walk $R$ must return to $B(0,2\varepsilon n)$ after leaving $B(0,\sqrt{\varepsilon} n - \varepsilon n)$. 
Thus,
\begin{multline*}
\P^0\left[\ql\left(\varepsilon n,\sqrt{\varepsilon}n;~\loe(R[0,T_{0,n}])\right)\cap B(0,\varepsilon n)\neq\emptyset\right]\\
\leq \P^0\left[R[T_{0,\sqrt{\varepsilon} n - \varepsilon n},~T_{0,n}]\cap B(0,2\varepsilon n)\neq \emptyset\right]
\leq C\sqrt{\varepsilon}. 
\end{multline*}
This implies \eqref{eq:ql} for quasi-loops at a vertex in $B(0,\varepsilon n)$. 

\medskip

Consider the second case. If $\ql\left(\varepsilon n,\sqrt{\varepsilon}n;~\loe(R[0,T_{0,n}])\right)\setminus B(0,n-\varepsilon n)\neq\emptyset$, 
then the random walk $R$ will hit $\partial B(0,n-2\varepsilon n)$ and then moves to distance $\sqrt{\varepsilon} n - 3\varepsilon n$ away from 
the hitting point before it exits from $B(0,n)$. Thus, there exist $\alpha>0$ and $C<\infty$ such that 
\begin{multline*}
\P^0\left[\ql\left(\varepsilon n,\sqrt{\varepsilon}n;~\loe(R[0,T_{0,n}])\right)\setminus B(0,n - \varepsilon n)\neq\emptyset\right]\\
\leq \P^0\left[R[T_{0,n-2\varepsilon n},~T_{0,n}]\cap \partial B(R(T_{0,n-2\varepsilon n}),\sqrt{\varepsilon}n-3\varepsilon n)\neq\emptyset\right]
\leq C\varepsilon^\alpha,
\end{multline*}
where the last inequality follows, for instance, from \cite[Lemma~2.5]{K}. 
This implies \eqref{eq:ql} for quasi-loops at a vertex outside of $B(0,n- \varepsilon n)$. 

\medskip

It remains to consider the third case. Let $A = B(0,n-\varepsilon n)\setminus \overline{B(0,\varepsilon n)}$.
We will prove that for some $M\geq 1$, 
\begin{equation}\label{eq:ql:case3}
\P^0\left[
\ql\left(\varepsilon^Mn,\varepsilon n;~\loe(R[0,T_{0,n}])\right)\cap A\neq\emptyset
\right]\leq C\varepsilon,
\end{equation}
which implies \eqref{eq:ql} for quasi-loops at a vertex in $A$. 
We cover $A$ by $s= 10\,\lfloor \varepsilon^{-6}\rfloor$ balls of radius $\varepsilon^2 n$ with centers at $v_1,\ldots, v_s\in A$. 
Then, the probability in \eqref{eq:ql:case3} is bounded from above by 
\[
10\,\varepsilon^{-6}\,\max_i 
\P^0\left[
\ql\left(\varepsilon^Mn,\varepsilon n;~\loe(R[0,T_{0,n}])\right)\cap B(v_i,\varepsilon^2 n)\neq\emptyset
\right],
\]
and the inequality \eqref{eq:ql:case3} is immediate from the following lemma. 
\begin{lemma}\label{l:ql:case3}
There exist $M,C<\infty$ such that for all $v\in A$, $\varepsilon>0$ and $n\geq1$, 
\[
\P^0\left[
\ql\left(\varepsilon^Mn,\varepsilon n;~\loe(R[0,T_{0,n}])\right)\cap B(v,\varepsilon^2 n)\neq\emptyset
\right] \leq C\varepsilon^7.
\]
\end{lemma}
The proof of Theorem~\ref{thm:ql} is complete subject to Lemma~\ref{l:ql:case3}, which we prove below. 
\end{proof}

\medskip

\begin{proof}[Proof of Lemma~\ref{l:ql:case3}]
Fix $v\in A$ and assume that $M>2$. Consider the stopping times $T_0 = 0$, 
\[
\begin{array}{rll}
T_{2k-1} &= &\inf\{t> T_{2k-2}~:~R(t)\in B(v,2\varepsilon^2 n)\},\\
T_{2k} &= &\inf\{t> T_{2k-1}~:~ R(t)\in\partial B(v,\frac12\varepsilon n)\},
\end{array}
\]
where $\inf\emptyset = \infty$. 
\begin{claim}\label{cl:Ti}
If $\ql\left(\varepsilon^Mn,\varepsilon n;~\loe(R[0,T_{0,n}])\right)\cap B(v,\varepsilon^2 n)\neq\emptyset$, then 
there exists $i$ such that 
\begin{equation}\label{eq:ql:i}
\begin{array}{rll}\itemsep3pt
\ql\left(\varepsilon^Mn,\varepsilon n;~\loe(R[0,T_{2i-1}])\right)\cap B(v,\varepsilon^2 n) &= &\emptyset, \\ 
\ql\left(\varepsilon^Mn,\varepsilon n;~\loe(R[0,T_{2i}])\right)\cap B(v,\varepsilon^2 n) &\neq &\emptyset.
\end{array}
\end{equation}
\end{claim}
\begin{proof}[Proof of Claim~\ref{cl:Ti}]
Assume that $\ql\left(\varepsilon^Mn,\varepsilon n;~\loe(R[0,T_{0,n}])\right)$ has a point in $B(v,\varepsilon^2 n)$. 
Let $I$ be such that $T_{2I-1} < T_{0,n} < T_{2I}$. 

Note that $\ql\left(\varepsilon^Mn,\varepsilon n;~\loe(R[0,T_{2I}])\right)$ also has a point in $B(v,\varepsilon^2 n)$, 
since $R[T_{2I},T_{0,n}]$ does not intersect $B(v,\varepsilon^2n + \varepsilon^Mn)$. 
If $\ql\left(\varepsilon^Mn,\varepsilon n;~\loe(R[0,T_{2I-1}])\right)\cap B(v,\varepsilon^2 n)=\emptyset$, 
then we are done, thus assume the contrary. 

Let $i\leq I$ and assume that $\ql\left(\varepsilon^Mn,\varepsilon n;~\loe(R[0,T_{2i-1}])\right)$ has a point in $B(v,\varepsilon^2 n)$. 
Since $R[T_{2i-2},T_{2i-1}]$ does not intersect $B(v,\varepsilon^2 n + \varepsilon^M n)$, the set 
$\ql\left(\varepsilon^Mn,\varepsilon n;~\loe(R[0,T_{2i-2}])\right)$ must also have a point in $B(v,\varepsilon^2 n)$. 
In this case, if $\ql\left(\varepsilon^Mn,\varepsilon n;~\loe(R[0,T_{2i-3}])\right)$ does not have a point in $B(v,\varepsilon^2 n)$, 
then we are done, otherwise, we repeat. 

Note that the above procedure eventually succeeds, since $\ql\left(\varepsilon^Mn,\varepsilon n;~\loe(R[0,T_1])\right)$ 
does not have a point in $B(v,\varepsilon^2 n)$. (In fact, $\ql\left(\varepsilon^Mn,\varepsilon n;~\loe(R[0,T_3])\right)$ too.)
\end{proof}

\medskip

We proceed to estimate the probability of event \eqref{eq:ql:i} for any fixed $i$. 
Let $i$ be fixed so that $T_{2i-1}<\infty$ and define $L' = \loe(R[0,T_{2i-1}])$ and $R' = R[T_{2i-1},T_{2i}]$. 
Consider the stopping times $\tau_0 = 0$, 
\[
\begin{array}{rll}
\tau_{2k-1} &= &\inf\{t> \tau_{2k-2}~:~L'(t)\in B(v,2\varepsilon^2 n)\},\\
\tau_{2k} &= &\inf\{t> \tau_{2k-1}~:~ L'(t)\in\partial B(v,\frac12\varepsilon n)\}.
\end{array}
\]
\begin{claim}\label{cl:taui}
If the event \eqref{eq:ql:i} occurs, then there exists $k$ and $x\in B(v,\varepsilon^2 n)$ such that 
\begin{equation}\label{eq:taui}
L'[0,\tau_{2k}]\cap B(x,\varepsilon^M n)\neq \emptyset, \qquad 
R'\cap B(x,\varepsilon^M n)\neq \emptyset,\qquad
R'\cap L'[0,\tau_{2k}]=\emptyset.
\end{equation}
\end{claim}
\begin{proof}[Proof of Claim~\ref{cl:taui}]
Assume that the event \eqref{eq:ql:i} occurs. 
Since $\loe(R[0,T_{2i}]) = \loe(L'\cup R')$ and $R'\subset \overline{B(v,\frac12\varepsilon n)}$, there exist $x\in B(v,\varepsilon^2 n)$ such that 
$L'[0,\mathrm{len}\,L']\cap B(x,\varepsilon^Mn)\neq\emptyset$ and $R'\cap B(x,\varepsilon^Mn)\neq\emptyset$. 
Otherwise, $R'$ would not complete any $(\varepsilon^Mn,\varepsilon n)$-quasi-loop in $B(v,\varepsilon^2 n)$. 

Let $k$ be the smallest index such that for some $x\in B(v,\varepsilon^2n)$, 
$L'[0,\tau_{2k}]\cap B(x,\varepsilon^M n)\neq \emptyset$ and $R'\cap B(x,\varepsilon^M n)\neq \emptyset$. 
We claim that $k$ satisfies \eqref{eq:taui}, i.e., $R'\cap L'[0,\tau_{2k}]=\emptyset$. 

Assume that $R'\cap L'[0,\tau_{2k}]\neq\emptyset$. 
Let $s$ be the smallest number that $L'[0,s]\cap R'\neq\emptyset$. By the assumption, $\tau_{2k}\geq s$.
Thus, $L'[\tau_{2k-1},s]\subset \overline{B(v,\frac12\varepsilon n)}$. 
Since also $R'\subset \overline{B(v,\frac12\varepsilon n)}$, 
there must exist $x\in B(v,\varepsilon^2 n)$ such that 
$L'[0,\tau_{2k-1}]\cap B(x,\varepsilon^M n)\neq\emptyset$ and $R'\cap B(x,\varepsilon^M n)\neq\emptyset$. 
Since $L'[\tau_{2k-2},\tau_{2k-1}]\cap B(v,\varepsilon^2n+\varepsilon^Mn)=\emptyset$, the above conclusion actually gives that 
for some $x\in B(v,\varepsilon^2 n)$, 
$L'[0,\tau_{2k-2}]\cap B(x,\varepsilon^M n)\neq\emptyset$ and $R'\cap B(x,\varepsilon^M n)\neq\emptyset$. 
This contradicts the minimality of $k$. 
\end{proof}

\medskip

Next we use Lemma~\ref{l:4.6} to prove that each $\loe(R[0,T_{2i-1}])[0,\tau_{2k}]$ is hittable by a random walk starting nearby. 

\begin{claim}\label{cl:ql:beurling}
There exist $M$ and $C$ such that for each $v\notin B(0,\varepsilon n)$, integers $i$ and $k$, and $\varepsilon>0$, 
if $L_{i,k} = \loe(R[0,T_{2i-1}])[0,\tau_{2k}]$, then 
\begin{equation}\label{eq:ql:beurling}
\P^0 \left[
\begin{array}{c}
\text{ $T_{2i-1}<\infty$, $\tau_{2k}<\infty$, and there exists  $x \in B(v,2\varepsilon^2 n)$ such that }\\[5pt]
\text{$\mathrm{dist}\left(x,L_{i,k}\right)\leq 2\varepsilon^M n$ and }\P^{2,x} \left[ R^2[0, T^2_{x,\frac14\varepsilon n}] \cap L_{i,k}= \emptyset \right] > \varepsilon^7
\end{array}
\right]
\leq C\,\varepsilon^7.
\end{equation}
\end{claim}
\begin{proof}[Proof of Claim~\ref{cl:ql:beurling}]
Assume that $\varepsilon$ is small (possibly depending on $M$). 
We cover $B(v,2\varepsilon^2 n)$ by $s = 10\,\lfloor \varepsilon^{-3M}\rfloor$ balls of radius $\varepsilon^Mn$ 
with centers at $v_1,\ldots, v_s\in B(v,2\varepsilon^2 n)$. 
Then the probabilty in \eqref{eq:ql:beurling} is bounded from above by 
\[
10\,\varepsilon^{-3M}\,\max_j
\P^0 \left[
\begin{array}{c}
\text{ $T_{2i-1}<\infty$, $\tau_{2k}<\infty$, and there exists  $x \in B(v_j,\varepsilon^M n)$ such that }\\[5pt]
\text{$\mathrm{dist}\left(x,L_{i,k}\right)\leq 2\varepsilon^M n$ and }\P^{2,x} \left[ R^2[0, T^2_{x,\frac14\varepsilon n}] \cap L_{i,k}= \emptyset \right] > \varepsilon^7
\end{array}
\right].
\]
For each $x\in  B\left(v_j,\varepsilon^M n\right)$, 
$B(x,2\varepsilon^M n)\subset B(v_j,3\varepsilon^M n)$ and 
$B(x,\frac14\varepsilon n)\supset B(v_j,\frac18\varepsilon n)$. 
Thus, the $j$th probability in the above maximum is at most  
\[
\P^0 \left[
\begin{array}{c}
\text{ $T_{2i-1}<\infty$, $\tau_{2k}<\infty$,}\\[5pt]
\text{$L_{i,k}\cap B(v_j,3\varepsilon^Mn)\neq\emptyset$, there exists  $x \in B(v_j,\varepsilon^M n)$ such that}\\[5pt]
\P^{2,x} \left[ R^2[0, T^2_{v_j,\frac18\varepsilon n}] \cap L_{i,k}\cap\left(B(v_j,\frac18\varepsilon n)\setminus\overline{B(v_j,3\varepsilon^Mn)}\right)= \emptyset \right] > \varepsilon^7
\end{array}
\right].
\]
By the Harnack inequality applied to the harmonic function
\[
\P^{2,x} \left[ R^2[0, T^2_{v_j,\frac18\varepsilon n}] \cap L_{i,k}\cap\left(B(v_j,\frac18\varepsilon n)\setminus\overline{B(v_j,3\varepsilon^Mn)}\right)= \emptyset \right],\qquad
x \in B\left(v_j,2\varepsilon^M n\right), 
\]
there exists a universal constant $c_*>0$ such that the $j$th probability is bounded from above by 
\[
\P^0 \left[
\begin{array}{c}
\text{ $T_{2i-1}<\infty$, $\tau_{2k}<\infty$, $L_{i,k}\cap B(v_j,3\varepsilon^Mn)\neq\emptyset$, and}\\[5pt]
\P^{2,v_j} \left[ R^2[0, T^2_{v_j,\frac18\varepsilon n}] \cap L_{i,k}\cap\left(B(v_j,\frac18\varepsilon n)\setminus\overline{B(v_j,3\varepsilon^Mn)}\right)= \emptyset \right] > c_*\varepsilon^7
\end{array}
\right].
\]
We apply Lemma~\ref{l:4.6} to $\loe(R[0,T_{2i-1}])[0,\tau_{2k}]$ with $v=v_j$, $r = \frac18\varepsilon n$, $s=3\varepsilon^Mn$, 
choosing $M$ so that $\left(\frac sr\right)^\eta = \left(24\,\varepsilon^{M-1}\right)^\eta<c_*\varepsilon^7$, 
and choosing $K$ so that $\left(\frac sr\right)^K = \left(24\,\varepsilon^{M-1}\right)^K < \varepsilon^{7+3M}$. 
This gives that the above probability is $\leq C\varepsilon^{7+3M}$. 
Thus, the original probability is $\leq (C\varepsilon^{7+3M})\cdot s \leq C'\varepsilon^7$. 
\end{proof}

\medskip

Let $E_i$ be the event in \eqref{eq:ql:i} and $E_{i,k}$ the event in \eqref{eq:taui} with $M$ as in Claim~\ref{cl:ql:beurling}. 
By Claim~\ref{cl:taui}, $\P^0[E_i]\leq i\,\max_k\P^0[E_{i,k}]$. 
On the event $E_{i,k}$, define the stopping time 
\[
\sigma_{i,k} = \inf\left\{t>T_{2i-1}~:~R(t)\in B(v,2\varepsilon^2n),~\mathrm{dist}(R(t),L_{i,k})\leq 2\varepsilon^M n\right\}.
\]
By the definition of $E_{i,k}$, $\sigma_{i,k}<T_{2i}$ and $R[\sigma_{i,k},T_{2i}]\cap L_{i,k} = \emptyset$. 
By Claim~\ref{cl:ql:beurling} and the strong Markov property, if $F_{i,k}$ is the event in \eqref{eq:ql:beurling}, then 
\[
\P^0[E_{i,k}]\leq \P^0[F_{i,k}] + \varepsilon^7 \leq (C+1)\varepsilon^7. 
\]
Thus, there exists $C$ such that for all $i$, $\P^0[E_i]\leq C\,i\,\varepsilon^7$. 
To conclude the proof, we notice that for each $i\geq 7$, 
\[
\P^0\left[T_{2i-1}<\infty\right]\leq C\,\varepsilon^i\leq C\,\varepsilon^7.
\]
By Claim~\ref{cl:Ti}, 
\[
\P^0\left[
\ql\left(\varepsilon^Mn,\varepsilon n;~\loe(R[0,T_{0,n}])\right)\cap B(v,\varepsilon^2 n)\neq\emptyset\right]
\leq \P^0\left[T_{13}<\infty\right] + \sum_{i=1}^7\P^0[E_i] \leq C\,\varepsilon^7.
\]
The proof of Lemma~\ref{l:ql:case3} is complete. 
\end{proof}

\bigskip

Exactly the same argument as in the second case in the proof of Theorem~\ref{thm:ql} gives the following result. 
\begin{proposition}\label{prop:bql}
There exist $\alpha>0$ and $C<\infty$ such that for all $\varepsilon>0$ and $n\geq 1$, if $L = \loe(R[0,T_{0,n}])$ then
\[
\P^0\left[
\begin{array}{c}
\text{There exist $i$ and $j$ such that $L(i)\notin B(0,n-\varepsilon n)$, }\\ 
\text{$L(j)\notin B(0,n-\varepsilon n)$, and $L[i,j]\nsubseteq B(L(i),\sqrt{\varepsilon}n)$}
\end{array}
\right] \leq C\,\varepsilon^\alpha.
\]
\end{proposition}

\subsection{Proof of Theorem~\ref{thm:simplepath}}\label{sec:simplepath:proof}

The proof of the theorem follows from Theorem~\ref{thm:ql} and Proposition~\ref{prop:bql} similarly to \cite[Theorem~1.1]{Sch}. 
Recall that a compact subset of $\R^d$ is a simple path if it is homeomorphic to $[0,1]$. 
We begin with the following observation, which is analogous to \cite[Theorem~3.5]{Sch}. 

For a simple path $\gamma$ and $x,y\in\gamma$, we denote by $D(x,y;\gamma)$ the diameter (with respect to the Euclidean distance) 
of the arc of $\gamma$ joining $x$ and $y$. For a function $f:(0,\infty)\to(0,\infty)$, let $\Gamma(f)$ be the set of paths from $\Gamma$ such that 
for all $x,y\in\gamma$, 
\[
\begin{array}{rll}
|x-y| &\geq &f(D(x,y;\gamma)),\\ 
\max\left(\mathrm{dist}(x,\partial D),\mathrm{dist}(y,\partial D)\right) &\geq &f(D(x,y;\gamma)),
\end{array}
\]
and denote by $\overline{\Gamma(f)}$ the closure of $\Gamma(f)$ in $(\mathcal K_D, d_H)$. 
\begin{lemma}\label{l:Gammaf}
For any monotone increasing and continuous function $f:(0,\infty)\to(0,\infty)$, 
\[
\overline{\Gamma(f)}\subset \Gamma.
\]
\end{lemma}
\begin{proof}[Proof of Lemma~\ref{l:Gammaf}]
Let $\gamma\in\overline{\Gamma(f)}$. Then $\gamma$ is a compact, connected subset of $\overline D$, such that $0\in\gamma$ and $\gamma\cap\partial D\neq\emptyset$. 

We first show that $|\gamma\cap\partial D| = 1$. Assume that there exist two different points $p,q\in\gamma\cap\partial D$. 
Let $\gamma_n\in\Gamma(f)$ be such that $d_H(\gamma_n,\gamma)<\frac1n$ and $p_n,q_n\in\gamma_n$ such that $|p_n-p|<\frac1n$ and $|q_n-q|<\frac1n$. 
In particular, $\mathrm{dist}(p_n,\partial D)<\frac1n$ and $\mathrm{dist}(q_n,\partial D)<\frac1n$. 
Since $\gamma_n\in\Gamma(f)$, 
\[
\frac1n > \max\left(\mathrm{dist}(p_n,\partial D),\mathrm{dist}(q_n,\partial D)\right)\geq f(D(p_n,q_n;\gamma_n))\geq f(|p_n-q_n|).
\]
By passing to the limit, we conclude that $p=q$. Thus, $|\gamma\cap\partial D| = 1$. We denote this point by $b$. 

\medskip

It remains to show that $\gamma$ is a simple path with $\gamma^e = b$. 
We will use the following Janiszewski's topological characterization of arcs (see, e.g., \cite[Lemma~3.6]{Sch}):
\begin{equation}\label{eq:Janiszewski}
\begin{array}{c}
\text{Let $\gamma$ be a compact, connected metric space, $a,b\in\gamma$.}\\
\text{If $\gamma\setminus\{x\}$ is disconnected for all $x\in\gamma\setminus\{a,b\}$, then 
$\gamma$ is a simple path from $a$ to $b$.}
\end{array}
\end{equation}

\medskip

Let $x\in\gamma\setminus\{0,b\}$. Let $\gamma_n\in\Gamma(f)$ be such that $d_H(\gamma_n,\gamma)<\frac1n$, 
and $x_n\in\gamma_n$ such that $|x_n-x|<\frac1n$. 
For each $n$, let $\gamma_n^1$ be the closed arc of $\gamma_n$ connecting $0$ and $x_n$, and 
$\gamma_n^2$ the closed arc of $\gamma_n$ connecting $z_n$ and $\gamma_n^e$. 
By passing to a subsequence, if necessary, assume with no loss of generality that 
$\gamma_n^1$ converges to $\gamma^1$ and $\gamma_n^2$ converges to $\gamma^2$ in $(\mathcal K_D, d_H)$. 
Note that $\gamma^1$ and $\gamma^2$ are compacts with $\gamma^1\cup\gamma^2 = \gamma$. 
For any $p\in\gamma^1$ and $q\in\gamma^2$, let $p_n\in\gamma_n^1$ be a sequence converging to $p$, 
and $q_n\in\gamma_n^2$ a sequence converging to $q$. 
Since $\gamma_n\in\Gamma(f)$, 
\[
|p_n-q_n|\geq f(D(p_n,q_n;\gamma_n))\geq f(|x_n-q_n|).
\]
By passing to the limit, we get that $|p-q|\geq f(|x-q|)$, for all $p\in\gamma^1$ and $q\in\gamma^2$, 
which implies that $\gamma^1\setminus\{x\}$ and $\gamma^2\setminus\{x\}$ are disjoint.
Therefore, $\gamma\setminus\{x\}$ is disconnected. 
By \eqref{eq:Janiszewski}, $\gamma$ is a simple path from $0$ to $b$. 

\medskip

We have shown that $\gamma\in \Gamma$.
\end{proof}

\medskip

The main ingredient for the proof of Theorem~\ref{thm:simplepath} is the following lemma. 
\begin{lemma}\label{l:lewfm}
The exist $C,M<\infty$ and $\alpha>0$ such that for all $m,n\geq 1$, 
\[
\P\left[\lew_n\in\Gamma(f_m)\right]\geq 1 - C\,2^{-\alpha m},
\]
where 
\[
f_m (s) = \min\left(2^{-Mm},~\left(\frac s4\right)^{2M}\right).
\]
\end{lemma}
\begin{proof}[Proof of Lemma~\ref{l:lewfm}]
Let $M$ be as in Theorem~\ref{thm:ql}. 
Let $n\geq 1$. 
Denote by $L = \loe(R[0,T_{0,n}])$. For $\varepsilon>0$, consider the events
\[
\begin{array}{rll}
E_\varepsilon &= &\left\{\ql(\varepsilon^Mn,\sqrt{\varepsilon} n,L)\neq\emptyset\right\},\\[3pt]
F_\varepsilon &= &\left\{\text{For some $i$ and $j$, $L(i),L(j)\notin B(0,n-\varepsilon n)$, $L[i,j]\nsubseteq B(L(i),\sqrt{\varepsilon}n)$}\right\},
\end{array}
\]
and denote by $E_i = E_{2^{-i}}$ and $F_i = F_{2^{-i}}$. 
By Theorem~\ref{thm:ql} and Proposition~\ref{prop:bql}, there exist $C<\infty$ and $\alpha>0$ such that 
\[
\P\left[\bigcup_{i=m}^\infty E_i\cup F_i\right]\leq C\,2^{-\alpha m}.
\]
Therefore, it suffices to show that 
\[
\left\{\lew_n\notin \Gamma(f_m)\right\}\subseteq \bigcup_{i=m}^\infty E_i\cup F_i.
\]
Assume that $E_i$ does not occur. Let $x,y\in \lew_n$ with $2^{-M(i+1)}\leq |x-y|< 2^{-Mi}$. Then, 
\[
D(x,y; \lew_n) \leq 2\sqrt{2^{-i}} \leq 4\, |x-y|^{\frac{1}{2M}}.
\]
Assume that $F_i$ does not occur. Let $x,y\in\lew_n$ with 
\[
2^{-(i+1)} \leq \max\left(\mathrm{dist}(x,\partial D),\mathrm{dist}(y,\partial D)\right)< 2^{-i}.
\]
Then, 
\[
D(x,y;\lew_n)\leq 2\,\sqrt{2^{-i}} \leq 4\,\max\left(\mathrm{dist}(x,\partial D),\mathrm{dist}(y,\partial D)\right)^{\frac12}.
\]
Thus, if $\bigcup_{i=m}^\infty E_i\cup F_i$ does not occur, then for all $x,y\in \lew_n$ with $|x-y|< 2^{-Mm}$ or 
$\max\left(\mathrm{dist}(x,\partial D),\mathrm{dist}(y,\partial D)\right)<2^{-Mm}$, 
\[
\min\left(|x-y|,\max\left(\mathrm{dist}(x,\partial D),\mathrm{dist}(y,\partial D)\right)\right)
\geq \left(\frac{D(x,y;\lew_n)}{4}\right)^{2M}.
\]
This implies that if $\bigcup_{i=m}^\infty E_i\cup F_i$ does not occur, then for all $x,y\in\lew_n$, 
\[
\min\left(|x-y|,\max\left(\mathrm{dist}(x,\partial D),\mathrm{dist}(y,\partial D)\right)\right)
\geq f_m\left(D(x,y;\lew_n)\right).
\]
Since $\lew_n\in\Gamma$, we have proved that $\lew_n\in \Gamma(f_m)$. 
The proof of Lemma~\ref{l:lewfm} is complete.
\end{proof}

\medskip

Theorem~\ref{thm:simplepath} easily follows from Lemmas~\ref{l:Gammaf} and \ref{l:lewfm}. 
\begin{proof}[Proof of Theorem~\ref{thm:simplepath}]
Let $f_m$ be as in Lemma~\ref{l:lewfm}. By Lemma~\ref{l:lewfm}, there exist $C<\infty$ and $\alpha>0$ such that for all $m,n\geq 1$, 
$\P\left[\lew_n\in\Gamma(f_m)\right]\geq 1 - C\,2^{-\alpha m}$. 
By the Portmanteau theorem (see, e.g., \cite[Theorem~2.1]{Billingsley}), for each $m\geq 1$, 
\[
\P\left[K\in\overline{\Gamma(f_m)}\right]\geq \limsup_{n\to\infty}\P\left[\lew_n\in\overline{\Gamma(f_m)}\right]
\geq \limsup_{n\to\infty}\P\left[\lew_n\in\Gamma(f_m)\right]\geq 1 - C\,2^{-\alpha m}.
\]
Since $f_m$ is increasing and continuous, by Lemma~\ref{l:Gammaf}, for all $m\geq 1$. 
\[
\P\left[K\in\Gamma\right]\geq \P\left[K\in\overline{\Gamma(f_m)}\right]\geq 1 - C\,2^{-\alpha m}.
\]
Thus, $\P\left[K\in\Gamma\right]=1$. 
\end{proof}

\section{Hausdorff dimension}\label{sec:Hausdorffdimension}

\subsection{Preliminaries on loop erased walks}\label{sec:Hausdorffdimension:preliminaries}

In this section we collect some auxiliary results about loop erased random walk. Let 
\[
\es(m,n) = \P^{1,0}\otimes\P^{2,0}\left[\loe(R^1[0,\infty))[s_m,t_n]\cap R^2[1,T^2_{0,n}] = \emptyset\right],
\]
where $t_n$ is the first time that $\loe(R^1[0,\infty))$ exits from $B(0,n)$, and 
$s_m$ is its last visit to $B(0,m)$ before time $t_n$.
We define $\es(m,n) = 1$ for $m\geq n$. 
Let
\[
\alpha = 2 - \beta \in [\frac13,1),
\]
where $\beta = \lim_{n\to\infty}\frac{\log \E^{0,1}\left[\mathrm{len}\,\loe(R^1[0,T^1_{0,n}])\right]}{\log n}$ is the growth exponent of the loop erased random walk. 
Its existence is shown in \cite{Shi13}. 
By \cite[Lemma~8.1.1]{Shi13}, for all $\epsilon>0$ there exist $C_{1,\epsilon},C_{2,\epsilon}\in(0,\infty)$ such that for all $m,n$, 
\begin{equation}\label{eq:es:bounds}
C_{1,\epsilon}\,\left(\frac mn\right)^{\alpha + \epsilon} \leq \es(m,n) \leq C_{2,\epsilon}\,\left(\frac mn\right)^{\alpha - \epsilon}.
\end{equation}
The following lemma is the main ingredient in the proof of the upper bound on the Hausdorff dimension of $\koz$. 
In its proof we will only use the upper bound from \eqref{eq:es:bounds}. 
\begin{lemma}\label{l:ball:proba}
For any $\delta>0$ there exists $C_\delta<\infty$ such that for all $\varepsilon>0$, $n\geq 1$ and $x\in B(0,n-2\varepsilon n)\setminus \overline{B(0,2\varepsilon n)}$, 
\begin{equation}\label{eq:ball:proba}
\P^{1,0}\left[\loe(R^1[0,T^1_{0,n}])\cap B(x,\varepsilon n)\neq \emptyset\right]\leq C_\delta\cdot \left(\frac{\varepsilon n}{|x|}\right)^{1+\alpha - \delta}.
\end{equation}
\end{lemma}
\begin{proof}
We write $B = B(x,\varepsilon n)$ throughout the proof. 
Let $L = \sup\{i\leq T^1_{0,n}:R^1(i)\in B\}$ be the time of last visit of $R^1$ to $B$ before leaving $B(0,n)$. 
We will use the following observation,
\[
\left\{\loe(R^1[0,T^1_{0,n}])\cap B\neq \emptyset\right\} = \left\{L<T^1_{0,n},~R^1[L+1,T^1_{0,n}]\cap\lambda(\loe(R^1[0,L])) = \emptyset \right\},
\]
where for a path $\gamma$, we write $\lambda(\gamma)$ for the piece of $\gamma$ from the start and until the first entrance to $B$. 
The probability of the event on the right hand side equals
\[
\sum_{y\in B}\sum_t\P^{1,0}\otimes \P^{2,y} \left[T^1_{0,n}>t,~R^1(t) = y,~ 
\lambda(\loe(R^1[0,t]))\cap R^2[1,T^2_{0,n}] = \emptyset,~R^2[1,T^2_{0,n}]\cap B = \emptyset\right].
\]
By the time reversibility of loop erasure, see \cite[Lemma~7.2.1]{L}, the probability in the sum equals to 
\[
\P^{1,y}\otimes \P^{2,y} \left[T^1_{0,n}>t,~R^1(t) = 0,~ 
\mu(\loe(R^1[0,t]))\cap R^2[1,T^2_{0,n}] = \emptyset,~R^2[1,T^2_{0,n}]\cap B = \emptyset\right],
\]
where for a path $\gamma$, we write $\mu(\gamma)$ for the piece of $\gamma$ from the last visit to $B$ until the end.
Summation over $t$ gives
\[
\P^{1,y}\otimes \P^{2,y} \left[T^1_{0,n}>T^1_0,~ 
\mu(\loe(R^1[0,T^1_0]))\cap R^2[1,T^2_{0,n}] = \emptyset,~R^2[1,T^2_{0,n}]\cap B = \emptyset\right],
\]
where $T^1_0 = T^1(\{0\})$. 
Let $d_x = \mathrm{dist}(x,\partial B(0,n))$ and define $R_x = \frac12\min(|x|,d_x)$. 
We first consider the case when $R^1$ returns to $B(y,2\varepsilon n)$ after leaving $B(y,\frac14 R_x)$:
\begin{multline*}
\sum_{y\in B}\P^{1,y}\otimes\P^{2,y}\left[R^1[T^1_{y,\frac14 R_x},T^1_{0,n}]\cap B(y,2\varepsilon n)\neq\emptyset,~T^1_{0,n}>T^1_0,~R^2[1,T^2_{0,n}]\cap B = \emptyset\right]\\
\stackrel{(*)}\leq C\cdot (\varepsilon n)^2\cdot \frac{\varepsilon n}{R_x}\cdot \frac{d_x}{|x|\,n}\cdot \frac{1}{R_x} 
\stackrel{(**)}\leq C'\cdot \left(\frac{\varepsilon n}{|x|}\right)^2
\leq C'\cdot \left(\frac{\varepsilon n}{|x|}\right)^{1 + \alpha},
\end{multline*}
where $(*)$ follows from \cite[Proposition~1.5.10]{L}, $(**)$ by considering separately the cases $R_x = |x|$ and $R_x = d_x$ and using $R_x\geq \varepsilon n$, 
and the last inequality from $\alpha\leq 1$. 

Thus, to establish \eqref{eq:ball:proba}, it remains to consider the case when $R^1$ does not return to $B(y,2\varepsilon n)$ after leaving $B(y,\frac14 R_x)$.
In this case, we have the inclusion 
\[
\left\{\mu(\loe(R^1[0,T^1_0]))\cap R^2[1,T^2_{0,n}] = \emptyset\right\}\subseteq 
\left\{\loe(R^1[0,T^1_0])[s,t]\cap R^2[1,T^2_{0,n}] = \emptyset\right\},
\]
where $t$ is the first time that $\loe(R^1[0,T^1_0])$ exits from $B(y,\frac14 R_x)$, and $s$ is its last visit time to $B(y,2\varepsilon n)$ before $t$. 
Therefore, in this case,
\begin{multline*}
\P^{2,y}\left[\mu(\loe(R^1[0,T^1_0]))\cap R^2[1,T^2_{0,n}] = \emptyset,~R^2[1,T^2_{0,n}]\cap B = \emptyset\right]\\
\leq 
\P^{2,y}\left[\loe(R^1[0,T^1_0])[s,t]\cap R^2[1,T^2_{0,n}] = \emptyset,~R^2[1,T^2_{0,n}]\cap B = \emptyset\right]\\
\stackrel{(*)}\leq
\P^{2,y}\left[R^2[1,T^2_{y,\varepsilon n}]\cap B = \emptyset\right]\cdot 
\max_{z\in\partial B(y,\varepsilon n)}\,\P^{2,z}\left[\loe(R^1[0,T^1_0])[s,t]\cap R^2[0,T^2_{0,n}] = \emptyset\right]\\
\stackrel{(**)}\leq
C\cdot \frac{1}{\varepsilon n}\cdot \P^{2,y}\left[\loe(R^1[0,T^1_0])[s,t]\cap R^2[0,T^2_{0,n}] = \emptyset\right],
\end{multline*}
where $(*)$ follows from the strong Markov property and $(**)$ from \cite[Proposition~1.5.10]{L} and the Harnack inequality. 
It remains to bound the probability
\[
\P^{1,y}\otimes \P^{2,y} \left[T^1_{0,n}>T^1_0,~ 
\loe(R^1[0,T^1_0])[s,t]\cap R^2[0,T^2_{0,n}] = \emptyset\right].
\]
By the results of \cite[Section~4.1]{M09}, if $X$ is a random walk from $y$ conditioned to hit $0$ before $\partial B(0,n)$ and killed upon hitting $0$, 
$Y$ is an infinite random walk from $y$, and the laws of their loop erasures until the first exit times from $B(y,\frac14 R_x)$ are $\P_X$ and $\P_Y$, 
then the Radon-Nikodym derivative $\frac{d\P_X}{d\P_Y}$ is bounded from above and below by universal positive and finite constants $C_1$ and $C_2$. 
Therefore, the probability in the above display is bounded from above by 
\[
C_2\cdot \P^{1,y}\otimes \P^{2,y} \left[
\loe(R^1[0,\infty))[s,t]\cap R^2[0,T^2_{0,n}] = \emptyset\right]\cdot \P^{1,y}\left[T^1_{0,n}>T^1_0\right],
\]
which is at most 
\[
C_2\cdot \es(2\varepsilon n,\frac14 R_x)\cdot \P^{1,y}\left[T^1_{0,n}>T^1_0\right]
\leq C_\delta\cdot \left(\frac{\varepsilon n}{R_x}\right)^{\alpha - \delta} \cdot \frac{d_x}{|x|\,n}
\leq C_\delta \cdot \frac{1}{\varepsilon n}\cdot \left(\frac{\varepsilon n}{|x|}\right)^{1+\alpha - \delta},
\]
where the last inequality again follows by considering cases $R_x=|x|$ and $R_x=d_x$ and using $\alpha\leq 1$. 
Finally, by summing over $y$ on the interior boundary of $B$, we get the bound
\[
C\cdot (\varepsilon n)^2\cdot \frac{1}{\varepsilon n}\cdot C_\delta\cdot \frac{1}{\varepsilon n}\cdot \left(\frac{\varepsilon n}{|x|}\right)^{1+\alpha - \delta},
\]
giving \eqref{eq:ball:proba}. The proof of Lemma~\ref{l:ball:proba} is complete. 
\end{proof}
\begin{remark}
Essentially the same proof gives the complementary lower bound to \eqref{eq:ball:proba} for $x$'s away from the boundary of $B(0,n)$.
For any $\delta>0$ there exists $c_\delta>0$ such that for all $\varepsilon>0$, $n\geq 1$ and $x\in B(0,\frac12 n)\setminus \overline{B(0,2\varepsilon n)}$, 
\[
\P^{1,0}\left[\loe(R^1[0,T^1_{0,n}])\cap B(x,\varepsilon n)\neq \emptyset\right]\geq c_\delta\cdot \left(\frac{\varepsilon n}{|x|}\right)^{1+\alpha + \delta}.
\]
\end{remark}

\subsection{Proof of Theorem~\ref{thm:hd}: upper bound}

In this section we use Lemma~\ref{l:ball:proba} to prove that $\hd(\koz)\leq 2 - \alpha$ almost surely. 
Recall that the Hausdorff dimension of a subset $S$ of $\R^d$ is defined as 
\[
\hd(S) = \inf\{\delta~:~\mathcal H^\delta(S) = 0\},
\]
where $\mathcal H^\delta(S)$ is the $\delta$-Hausdorff measure of $S$, 
$\mathcal H^\delta(S) = \lim_{\varepsilon\to 0}\inf\sum_{j=1}^\infty \mathrm{diam}(S_j)^\delta$, 
and the infimum is taken over all countable collections of sets $\{S_j\}$ covering $S$ with $\mathrm{diam}(S_j)\leq \varepsilon$.

Consider the coupling of $\koz$ and $\lew_n$ such that $d_H(\lew_n,\koz)\to 0$ as $n\to\infty$ almost surely. 
Let $N_1(\varepsilon)$ be the number of balls of radius $\frac12\varepsilon$ centered in $\frac12\varepsilon\Z^3\cap \left(D_{1-2\varepsilon}\setminus D_{2\varepsilon}\right)$
that have non-empty intersection with $\koz$, and  
$N_2(\varepsilon)$ the number of balls of radius $\frac12\varepsilon$ centered in $\frac12\varepsilon\Z^3\cap \left(D_{2\varepsilon}\cup D_{1-2\varepsilon}^c\right)$
that have non-empty intersection with $\koz$. 
Similarly, define $N_{1,n}(\varepsilon)$ as the number of balls of radius $\varepsilon n$ centered in $\frac12\varepsilon n\Z^3\cap B(0,n-2\varepsilon n)\setminus\overline{B(0,2\varepsilon n)}$ 
that have non-empty intersection with $\loe(R[0,T_{0,n}])$, 
and $N_{2,n}(\varepsilon)$ as the corresponding number of balls centered in $\frac12\varepsilon n\Z^3\cap \left(B(0,2\varepsilon n)\cup B(0,n-2\varepsilon n)^c\right)$. 
Then, for any positive $\delta$ and $\xi$, 
\[
\P\left[N_1(\varepsilon)\geq \delta\,\varepsilon^{\alpha-2-\xi}\right] \leq \P\left[d_H(\lew_n,\koz)\geq\frac12\varepsilon\right] + 
\P\left[N_{1,n}(\varepsilon) \geq \delta\,\varepsilon^{\alpha-2-\xi}\right].
\]
By Lemma~\ref{l:ball:proba}, 
\[
\P\left[N_{1,n}(\varepsilon) \geq \delta\,\varepsilon^{\alpha-2-\xi}\right]
\leq \frac1\delta\,\varepsilon^{2-\alpha+\xi}\,\E\left[N_{1,n}(\varepsilon)\right]\leq \frac1\delta\, C_\xi\,\varepsilon^{\frac12\xi}.
\]
By sending $n$ to infinity, we obtain that $\P\left[N_1(\varepsilon)\geq \delta\,\varepsilon^{\alpha-2-\xi}\right]\leq \frac1\delta\, C_\xi\,\varepsilon^{\frac12\xi}$. 
To obtain a bound for $N_2(\varepsilon)$, we proceed as above, but use Proposition~\ref{prop:bql} instead of Lemma~\ref{l:ball:proba}. 
We get $\P\left[N_2(\varepsilon)\geq \delta\,\varepsilon^{-\frac12}\right]\leq C_\delta\,\varepsilon^{\frac12\xi}$, 
if $\xi$ is sufficiently small. Since $\alpha\leq 1$, this implies that $\P\left[N_2(\varepsilon)\geq \delta\,\varepsilon^{\alpha-2-\xi}\right]\leq C_\delta\,\varepsilon^{\frac12\xi}$.

\medskip

For $\gamma\geq 0$, let $\mathcal H^\gamma_\varepsilon(\koz) = \inf\sum_{j=1}^\infty \mathrm{diam}(S_j)^\gamma$, 
where the infimum is taken over all coverings of $\koz$ by sets $S_j$ with diameter at most $\varepsilon$. 
Then, $\mathcal H^\gamma_\varepsilon(\koz)\leq \varepsilon^\gamma\,\left(N_1(\varepsilon)+N_2(\varepsilon)\right)$, and we obtain from the above estimates that 
\[
\P\left[\mathcal H^{2-\alpha+\xi}_\varepsilon(\koz)\geq 2\delta\right]\leq C_{\xi,\delta}\,\varepsilon^{\frac12\xi}.
\]
Note that if $\varepsilon\searrow 0$ then $\mathcal H^{2-\alpha+\xi}_\varepsilon(\koz)\nearrow\mathcal H^{2-\alpha+\xi}(\koz)$. 
Thus, for all $\delta>0$, $\P[\mathcal H^{2-\alpha+\xi}(\koz)\geq 2\delta]=0$, i.e., $\mathcal H^{2-\alpha+\xi}(\koz) = 0$ almost surely. 
Since $\xi>0$ is arbitrary, we get $\hd(\koz)\leq 2-\alpha$. 
\qed

\subsection{Proof of Theorem~\ref{thm:hd}: lower bound}

Let $\bm$ be the Brownian motion in $\R^3$ and $\tau = \inf\{t\geq 0:|B(t)|=1\}$ the first exit time of $\bm$ from $D$. 
The set of cut points $\mathrm{C}$ of $\bm[0,\tau]$ is defined as  
\[
\mathrm{C} = \{\bm(t)~:~0\leq t\leq\tau,~\bm[0,t]\cap\bm(t,\tau]=\emptyset\}.
\]
It is proved in \cite{L95} that $\hd(\mathrm{C}) = 2 - \xi$ almost surely, where $\xi$ is the non-intersection exponent for $3$ dimensional Brownian motion 
satisfying $\xi\in(\frac12,1)$.
Note that every path in $\bm[0,\tau]$ from $\bm(0)=0$ to $\bm(\tau)\in\partial D$ goes through all the cut points. 
We denote by $\mathcal{S}(U)$ the set of all points of $U\subseteq \overline D$ which disconnect $0$ from $\partial D$ in $U$. 
As noticed above, $\mathcal{S}(\bm[0,\tau])\supseteq\mathrm{C}$, thus, $\hd(\mathcal{S}(\bm[0,\tau])) \geq 2 - \xi$ almost surely. 

Now, recall from Theorem~\ref{thm:K+BLS=BM} that $\bm[0,\tau]$ has the same distribution as the union of 
the independent scaling limit of the loop erased random walk, $\koz$, and all the loops from the Brownian loop soup of intensity $1$ that 
are contained in $D$ and intersect $\koz$. Denote this union by $X$. 
Then $\mathcal{S}(X)$ has the same distribution as $\mathcal S(\bm[0,\tau])$ and, since $\koz$ connects $0$ and $\partial D$, $\mathcal S(X)\subseteq\koz$. 
Thus, $\hd(\koz) \geq 2 - \xi$ almost surely.
\qed

\paragraph{Acknowledgements.} Enormous thanks go to Alain-Sol Sznitman for his helpful discussions, encouragements, and fruitful comments. 
We also thank Yinshan Chang and Wendelin Werner for useful discussions and suggestions, 
and Yinshan Chang for a careful reading of the paper. 
This project was carried out while the second author was enjoying the hospitality of the 
Forschungsinstitut f\"ur Mathematik of the ETH Z\"urich and the Max Planck Institute for Mathematics in the Sciences. 
He wishes to thank these institutions. The reseach of the second author has been supported by the Japan Society for the Promotion of Science (JSPS). 
Finally, the second author thanks Hidemi Aihara for all her understanding and support.


\begin{thebibliography}{99}

\bibitem{BS03}
K. Ball and J. Sterbenz. Explicit bounds for the return probability of simple random walks. 
{\it Journal of Theoretical Probability} {\bf 18(2)}, 317--326. (2005)

\bibitem{B} 
V. Beffara. The dimension of the SLE curves. 
{\it Ann. Probab.} {\bf 36(4)}, 1421--1452. (2008)

\bibitem{Billingsley}
P. Billingsley. Convergence of probability measures. Wiley Series in Probability and Statistics. (1999) 

\bibitem{Cam15}
F. Camia. Brownian loops and conformal fields. see http://arxiv.org/abs/1501.04861

\bibitem{DEK} 
A. Dvoretzky, P. Erd\H{o}s and S. Kakutani. Double points of paths of Brownian motion in n-space. 
{\it Acta Sci. Math. Szeged} {\bf 12}, 75--81. (1950)

\bibitem{Fri}
B. Fristedt. An extension of a theorem of S. J. Taylor concerning the multiple points of the symmetric stable process. 
{\it Z. Wahrscheinlichkeitstheorie und Verw. Gebiete} {\bf 9}, 62--64. (1967)

\bibitem{GB} 
A. J. Guttmann and R. J. Bursill. Critical exponents for the loop erased self-avoiding walk by Monte Carlo methods. 
{\it Journal of Statistical Physics} {\bf 59(1)}, 1--9. (1990)

\bibitem{Kak} 
S. Kakutani. On Brownian motions in n-space. 
{\it Proc. Imp. Acad. Tokyo} {\bf 20}, 648--652 (1944)

\bibitem{Ken} 
R. Kenyon. The asymptotic determinant of the discrete Laplacian. 
{\it Acta Mathematica} {\bf 185(2)}, 239--286 (2000)

\bibitem{K87}
H. Kesten. Hitting probabilities of random walks on $\mathbb{Z}^{d}$.
{\it Stochastic Processes and their Applications} {\bf 25}, 165--184 (1987)

\bibitem{K}
G. Kozma. The scaling limit of loop-erased random walk in three dimensions.
{\it Acta Mathematica} {\bf 199(1)}, 29--152 (2007)

\bibitem{L}
G. F. Lawler. Intersections of random walks. Birkhauser, Boston, (1991)

\bibitem{Law0} 
G. F. Lawler. A self avoiding walk.
{\it Duke Math. J.} {\bf 47}, 655--694  (1980) 

\bibitem{L95}
G. F. Lawler. Hausdorff dimension of cut points for Brownian motion. 
{\it El. Journal Probab.} {\bf 1}, 1--20 (1996)

\bibitem{Law99} 
G. F. Lawler. Loop-erased random walk, in: Perplexing problems in probability. Progress in Probability {\bf 44}, Birkhauser Boston, 197--217 (1999)

\bibitem{Law-2} 
G. F. Lawler. The probability that planar loop-erased random walk uses a given edge.
{\it El. Journal Probab.} {\bf 19}, 1--13 (2014)

\bibitem{LL10}
G. F. Lawler and V. Limic. Random walk: a modern introduction. Cambridge Studies in Advanced Mathematics. (2010)

\bibitem{LSW} 
G. F. Lawler, O. Schramm and W. Werner. Conformal invariance of planar loop-erased random walks and uniform spanning trees.
{\it Ann. Probab.} {\bf 32(1B)}, 939--995 (2004) 

\bibitem{LSW2} 
G. F. Lawler, O. Schramm and W. Werner. Conformal restriction: the chordal case.
{\it Journal of the American Math. Soc.} {\bf 16(4)}, 917-955 (2003)

\bibitem{LT04}
G. F. Lawler and J. Trujillo Ferreras. Random walk loop soup.
{\it Transactions of the American Mathematical Society} {\bf 359(2)}, 767--787 (2007) 

\bibitem{LW04}
G. F. Lawler and W. Werner. The Brownian loop soup.
{\it Probab. Theory Related Fields} {\bf 128(4)}, 565--588 (2004)

\bibitem{LeJ08}
Y. Le Jan. \emph{Markov paths, loops and fields}, Lecture Notes in Mathematics,
  vol. 2026, Springer, Heidelberg, 2011, Lectures from the 38th Probability
  Summer School held in Saint-Flour, 2008.  

\bibitem{Lev} 
L\'evy. Le mouvement brownien plan. 
{\it Amer. J. Math.} {\bf 62}, 487--550 (1940)

\bibitem{M09}
R. Masson. The growth exponent for planar loop-erased random walk.
{\it El. Journal Probab.} {\bf 14}, 1012--1073 (2009)

\bibitem{Pem} 
R. Pemantle. Choosing a spanning tree for the integer lattice uniformly.
{\it Ann. Probab.} {\bf 19(4)}, 1559--1574 (1991)

\bibitem{Sch}
O. Schramm. Scaling limits of loop-erased random walks and uniform spanning trees.
{\it Israel Journal of Mathematics} {\bf 118(1)}, 221--288 (2000)

\bibitem{Shi13}
D. Shiraishi. Growth exponent for loop-erased random walk in three dimensions. see http://arxiv.org/abs/1310.1682

\bibitem{Shi:wip}
D. Shiraishi. Hausdorff dimension of the scaling limit of loop-erased random walk in three dimensions. In preparation.

\bibitem{Sla} 
G. Slade. Self-avoiding walks. 
{\it Math. Intelligencer} {\bf 16(1)}, 29--35 (1994)

\bibitem{Sym}
K.~Symanzik, \emph{Euclidean quantum field theory}, Scuola internazionale di
  Fisica "Enrico Fermi" \textbf{XLV} 152--223. (1969)

\bibitem{Szn12}
A.-S. Sznitman. \emph{Topics in occupation times and {G}aussian free fields}, Z\"urich
  Lectures in Advanced Mathematics, European Mathematical Society, (2012) 
  
\bibitem{Tay}
S. J. Taylor. Multiple points for the sample paths of the symmetric stable process.
{\it Zeitschr. Wahrsch. verw. Gebiete} {\bf 5}, 247--264 (1966)

\bibitem{Wil} 
D. B. Wilson. Generating random spanning trees more quickly than the cover time, 
Proceedings of the Twenty-eighth Annual ACM Symposium on the Theory of Computing (Philadelphia, PA, 1996), 296--303, ACM, New York, (1996)

\bibitem{Wil2} 
D. B. Wilson. The dimension of loop-erased random walk in 3D. 
{\it Physical Review E}  {\bf 82(6)}, 062102,  (2010)

\bibitem{Zhan}
D. Zhan. Loop-Erasure of Plane Brownian Motion.
{\it Communications in Mathematical Physics} {\bf 303(3)}, 709--720 (2011)

\end{thebibliography}
\end{document}